\documentclass[fleqn,leqno]{amsart}
\pdfoutput=1 \usepackage{amssymb}
\usepackage{array}
\usepackage{booktabs}
\usepackage{colonequals}
\usepackage{enumitem}
\usepackage{float}
\usepackage{geometry}
\usepackage{hyperref}
\usepackage{mathtools}
\usepackage{subcaption}
\usepackage{tikz-cd}
\usepackage{mathrsfs}
\usepackage{comment}

\usepackage[capitalise]{cleveref}

\usepackage[utf8]{inputenc}
\usepackage[T1]{fontenc}
\usepackage[backend=biber, maxbibnames=99, sortcites, texencoding=ascii]{biblatex}

\DeclareFieldFormat{mrnumber}{MR\addcolon\space
  \ifhyperref
    {\href{http://www.ams.org/mathscinet-getitem?mr=1#1}{\nolinkurl{#1}}}
    {\nolinkurl{#1}}}

\renewbibmacro*{doi+eprint+url}{\iftoggle{bbx:doi}
    {\printfield{doi}}
    {}\newunit\newblock
  \printfield{mrnumber}\newunit\newblock
  \iftoggle{bbx:eprint}
    {\usebibmacro{eprint}}
    {}\newunit\newblock
  \iftoggle{bbx:url}
    {\usebibmacro{url+urldate}}
    {}}

\DeclareFieldFormat[online]{title}{``#1''\isdot}

\theoremstyle{definition}
\newtheorem{theorem}{Theorem}[section]
\newtheorem{corollary}[theorem]{Corollary}

\newtheorem{lemma}[theorem]{Lemma}
\newtheorem{proposition}[theorem]{Proposition}
\newtheorem{definition}[theorem]{Definition}
\newtheorem{remark}[theorem]{Remark}
\newtheorem{example}[theorem]{Example}
\newtheorem*{claim*}{Claim}

\newtheorem{alphatheorem}{Theorem}[section]
\newtheorem{alphacorollary}[alphatheorem]{Corollary}

\newcommand\cc{\ensuremath{\mathrm{c}}}
\newcommand\dd{\ensuremath{\mathrm{d}}}
\newcommand\LLL{\ensuremath{\mathbf{L}}}
\newcommand\pt{\ensuremath{\mathrm{pt}}}
\newcommand\RRR{\ensuremath{\mathbf{R}}}
\newcommand\tangent{\ensuremath{\mathrm{T}}}
\newcommand\todd{\ensuremath{\mathrm{td}}}

\DeclareMathOperator\Aut{Aut}
\DeclareMathOperator\Bl{Bl}
\DeclareMathOperator\chern{ch}

\DeclareMathOperator\Ext{Ext}
\DeclareMathOperator\Gr{Gr}
\DeclareMathOperator\rank{rank}
\DeclareMathOperator\codim{codim}

\DeclareMathOperator\HH{H}
\DeclareMathOperator\HHHH{HH}
\DeclareMathOperator\hh{h}

\DeclareMathOperator\Lie{Lie}
\DeclareMathOperator\PGL{PGL}
\DeclareMathOperator\Pois{Pois}
\DeclareMathOperator\pv{p} \DeclareMathOperator\rk{rank}

\usepackage[disable]
{todonotes}

\newcommand\MM[2]{\ensuremath{\mathrm{MM}_{#1.#2}}}
\newcommand\parallelogramthreefold[6]{\begin{array}{*4{>{}c}} 1 \\ 0 & #1 \\ 0 & #2 & #3 \\ 0 & 0 & #4 & #6 \\ & 0 & #5 & 0 \\ & & 0 & 0 \\ & & & 0 \end{array}}
\newcommand\parallelogramsurface[3]{\begin{array}{*3{>{}c}} 1 \\ 0 & #1 \\ 0 & #2 & #3 \\ & 0 & 0 \\ & & 0 \end{array}}

\let\oldbigwedge\bigwedge
\renewcommand\bigwedge{\oldbigwedge\nolimits}

\renewcommand\paragraph[1]{\bgroup\setlength\parindent{0pt}\vspace{5pt}\textbf{#1.}\egroup}

\title{Polyvector fields for Fano 3-folds}

\author{Pieter Belmans}
\address{(1) Mathematisches Institut, Universit\"at Bonn, Endenicher Allee 60, 53115 Bonn, Germany (2) Departement Wiskunde en statistiek, Universiteit Antwerpen, Middelheimlaan 1, 2020 Antwerpen, Belgium, (3) Universiteit Hasselt, Agoralaan, 3590 Diepenbeek, Belgium}

\author{Enrico Fatighenti}
\address{Institut de math\'ematiques de Toulouse, Université Paul Sabatier, 118 route de Narbonne, 31062 Toulouse, France}

\author{Fabio Tanturri}
\address{Dipartimento di Matematica, Universit\`a di Genova, via Dodecaneso 35, 16146 Genova, Italy}

\begin{document}

\begin{abstract}
  We compute the Hochschild--Kostant--Rosenberg decomposition of the Hochschild cohomology of Fano 3-folds. This is the first step in understanding the non-trivial Gerstenhaber algebra structure of this invariant, and yields some initial insights in the classification of Poisson structures on Fano 3-folds of higher Picard rank.
\end{abstract}

\maketitle

\tableofcontents

\section{Introduction}
In this paper we describe the Hochschild cohomology of Fano 3-folds, with the eventual goal of understanding the interesting algebraic structures present on this invariant, and completing the classification of Poisson structures on Fano 3-folds.

\paragraph{Fano 3-folds and the vector bundle method}
Fano 3-folds were classified by Iskovskikh \cite{MR0463151,MR0503430} (for Picard rank~1, where there are 17 families)
and Mori--Mukai \cite{MR0641971,MR1102273} (for Picard rank~$\geq 2$, where there are 88 families).
This classification was obtained by understanding the \emph{birational geometry} of Fano 3-folds, and the output is a list of 105 deformation families and their numerical invariants~$\cc_1(X)^3$, $\rho(X)$, and~$\hh^{1,2}(X)$. Only~12 out of 88~families of Picard rank~$\geq 2$ are not the blowup of a Fano 3-fold of lower Picard rank.

For the Picard rank~1 case Mukai alternatively described the classification using the \emph{vector bundle method} in \cite{MR0995400}, by writing Fano 3-folds of Picard rank 1 as zero loci of vector bundles on homogeneous varieties and weighted projective spaces. In higher Picard ranks this was extended in 2 different ways, by giving a description as
\begin{enumerate}
  \item zero loci of vector bundles on GIT quotients by products of general linear groups \cite{MR3470714}; or
  \item zero loci of homogeneous vector bundles on homogeneous varieties \cite{2009.13382v1}.
\end{enumerate}
The ambient variety is often called the \emph{key variety}, and will be denoted~$F$.
In the first variation on the vector bundle method the group is often a product of tori, so that the Fano 3-fold is described as a complete intersection in a toric variety~$F$. There are~13~families for which the group is not a product of tori, and the key variety is actually a product of Grassmannians.

In the second variation the key variety~$F$ is always a homogeneous variety. In particular, every Fano 3-fold can be realised as such in a product of (possibly weighted) Grassmannians.

\paragraph{Hochschild cohomology}
We will use both these descriptions to determine the \emph{Hochschild cohomology} of all Fano 3-folds. This is an invariant which measures the deformation theory of the abelian category (or derived category) of (quasi)coherent sheaves \cite{MR2238922,MR2183254,MR2477894}. For the definition and more details on the algebraic structure on~$\HHHH^\bullet(X)$, see \cref{subsection:hochschild-cohomology}.

An important instrument in describing Hochschild cohomology for varieties is the Hochschild--Kostant--Rosenberg decomposition \cite{MR1390671,MR1940241,MR2141853}, which says that
\begin{equation}
  \label{equation:HKR-introduction}
  \HHHH^i(X)\cong\bigoplus_{p+q=i}\HH^p(X,\bigwedge^q\tangent_X).
\end{equation}

The birational description of Mori--Mukai is not convenient for automating computations of the right-hand side of \eqref{equation:HKR-introduction} for Fano 3-folds, whereas the vector bundle method turns out to be well-suited for this. Moreover, we need the combination of both descriptions to cover all Fano 3-folds, together with a separate analysis of some underdetermined cases, i.e.~cases for which the two approaches do not yield a complete description of the cohomology groups we are aiming at.

In \cref{subsection:general-results} we will show how to compute $\hh^p(X,\bigwedge^q\tangent_X)$ for $q \neq 2$: the summands with~$q=0$ and~$q=3$ are easy, and the summands for~$q=1$ follow from the knowledge of the invariants~$\mathrm{c}_1^3$, $\rho$, and $\mathrm{h}^{1,2}$ together with the size of the automorphism group of a Fano 3-fold $X$. For~$q=2$ the description is new, and forms the main subject of this paper.
\begin{alphatheorem}
  \label{theorem:main-theorem}
  Let~$X$ be a Fano 3-fold. Then the cohomology of~$\bigwedge^2\tangent_X$ is concentrated in degrees~$0,1,2$, and it is constant in families.
The dimensions of the cohomologies (for all~$q=0,1,2,3$) are given in the tables in \cref{section:values}.
\end{alphatheorem}
The fact that the dimension of cohomology is constant in families is a by-product of the calculations, we don't have an abstract proof for it. Observe that the cohomology of the tangent bundle is \emph{not} constant in families, see \cite{MR3985696} for the jumping behaviour of~$\Aut^0(X)$ and therefore~$\hh^0(X,\tangent_X)=\dim\Aut^0(X)$.

\paragraph{On the methods}
In \cref{section:computation,section:by-hand} we collect the details for the proof of \cref{theorem:main-theorem}. We will set up the proof so that we can take advantage of computer algebra methods, with some explicit calculations in cases where automated methods fail. We have optimised the automated methods so that only~5/105 deformation families of Fano 3-folds need to be dealt with by hand (2 of which are nearly immediate).

What is interesting to observe is that the homogeneous methods from \cite{2009.13382v1} are very good at determining Hodge numbers (and in particular they are expected to help in classifying Fano 4-folds), with only a dozen deformation families of Fano 3-folds not being fully determined. But for \emph{twisted} Hodge numbers (and in particular the cohomology of~$\tangent_X$ and~$\bigwedge^2\tangent_X$) the homogeneous approach gives many underdetermined cases.

This is why we first use the toric description from \cite{MR3470714}, and only use the homogeneous description when no such description is available or when the toric methods are not giving a full answer. The combination of these two methods, in this particular order, gives the cleanest exposition.

\paragraph{Absence of Poisson structures}
When~$\HH^0(X,\bigwedge^2\tangent_X)$ is non-zero, the classification of Poisson structures becomes an interesting question. For a global bivector, the vanishing of the self-bracket (for the Schouten bracket) is equivalent to the Jacobi identity of the associated Poisson structure.

In \cite[\S9, Table~1]{MR3066408} Poisson structures on Fano 3-folds of Picard rank 1 were classified (see also \cite{MR2129789} for the classification of Poisson structures on smooth projective surfaces, where the vanishing is automatic).
As an immediate corollary of \cref{theorem:main-theorem} we obtain the \emph{absence} of Poisson structures on some Fano 3-folds of higher Picard rank. Here the notation~$\MM{\rho}{n}$ refers to the $n$th deformation family with Picard rank~$\rho$ in the Mori--Mukai classification, see \cite{MR0641971} and \cite[\S12.2--12.6]{MR1668579}.

\begin{alphacorollary}
  \label{corollary:no-poisson}
  The following primitive\footnote{i.e. cannot be written as the blowup in a curve of a Fano 3-fold of lower Picard rank} Fano 3-folds with~$\rho\geq 2$ admit no Poisson structures:
\begin{itemize}
    \item \MM{2}{2}
    \item \MM{2}{6}
    \item \MM{3}{1}.
  \end{itemize}
  The following imprimitive Fano 3-folds admit no Poisson structures:
  \begin{itemize}
    \item \MM{2}{4}
    \item \MM{2}{7}
    \item \MM{3}{3}.
  \end{itemize}
\end{alphacorollary}
For all other Fano 3-folds there are non-zero global bivectors, and it is necessary to check the self-bracket of a global bivector field. Already for Fano 3-folds of Picard rank~1 this is a highly non-trivial condition \cite{MR3066408}. 

For the imprimitive Fano 3-folds we expect that the birational description of Mori--Mukai together with \cite[\S8]{MR1465521} should allow for a (partial) classification of Poisson structures. In particular, we expect that the second part of \cref{corollary:no-poisson} has a proof using these techniques, but this is outside the scope of the current paper.

\paragraph{Relation to other works}
In the representation theory of finite-dimensional algebras the Gerstenhaber algebra structure on Hochschild cohomology is an important invariant, studied in many cases, see \cite{
MR4186975, MR3359733, MR3748354, MR4097322} to name a few. In algebraic geometry there are (at the time of writing) fewer attempts at giving explicit descriptions of Hochschild cohomology and the Hochschild--Kostant--Rosenberg decomposition. An important case is that of partial flag varieties \cite{1911.09414v1,MR4187255}. For smooth projective toric varieties (and only the~$\HH^0$, not any possible~$\HH^{\geq1}$) one is referred to \cite{2010.07053v1}. There are also various cases where the interaction of the Hochschild cohomology of different varieties (and categories) is studied (see e.g.~\cite{MR3950704,MR3420334,MR3904800}), with the Kuznetsov components of Fano 3-folds of Picard rank 1 and index 2 being the subject of \cite[\S8.3]{0904.4330v1}.

Some of the results in this paper are standard, whilst for Fano 3-folds of Picard rank~1 results can be found in \cite{MR2290387,MR3066408}.

It would be interesting to understand how mirror symmetry can be used to compute the invariants investigated in this paper, using the symplectic geometry of the mirror Landau--Ginzburg model. For Hodge numbers (and hence Hochschild homology, see \cref{subsection:hochschild-cohomology}) of Fano varieties a recipe for this was conjectured by Katzarkov--Kontsevich--Pantev in \cite[Conjecture~3.7]{MR3592695}, based on the conjectural equivalence
\begin{equation}
  \label{equation:hms}
  \mathbf{D}^{\mathrm{b}}(X)\cong\mathrm{FS}(Y,f,\omega_Y)
\end{equation}
from homological mirror symmetry. Here~$f\colon Y\to\mathbb{A}^1$ is a (suitably compactified) Landau--Ginzburg model and~$\omega_Y$ an appropriately chosen symplectic form, so that~$X$ and~$(Y,f)$ are mirror. Subsequently this was checked by Lunts--Przyjalkowski for del Pezzo surfaces in \cite{MR3783413} and by Cheltsov--Przyjalkowski for Fano 3-folds in \cite{1809.09218v1}. Hochschild cohomology is also a categorical invariant, and therefore can be computed from either side of \eqref{equation:hms} (assuming an enhancement of the equivalence). An interesting difference is that Hodge numbers (and hence the dimensions of the Hochschild homology spaces) are constant in families, but this is not the case for Hochschild cohomology.

\paragraph{Notation}
We will number deformation families of Fano 3-folds as~\MM{\rho}{n} as in Mori--Mukai \cite{MR0641971} (see also \cite[\S12.2--12.6]{MR1668579}), with the caveat that~\MM{4}{13} refers to the blowup of~$\mathbb{P}^1\times\mathbb{P}^1\times\mathbb{P}^1$ in a curve of degree~$(1,1,3)$, the case which was originally omitted and discovered in \cite{MR1969009}.

Throughout we work over an algebraically closed field~$k$ of characteristic~0.

\paragraph{Acknowledgements}
We would like to thank Marcello Bernardara, Alexander Kasprzyk, Brent Pym and Helge Ruddat for interesting conversations. And we want to thank Alexander Kuznetsov for many conversations over the years about Fano 3-folds, and comments on an earlier version of this paper.

The first author was partially supported by the FWO (Research Foundation--Flanders). The second and third author are members of INdAM-GNSAGA.

\section{Polyector fields and their structure}
\subsection{Hochschild cohomology and the Hochschild--Kostant--Rosenberg decomposition}
\label{subsection:hochschild-cohomology}
There exist various approaches to defining the Hochschild cohomology of a variety, which are known to agree in the setting we are interested in. One of the more economical definitions is the following.
\begin{definition}
  Let~$X$ be a smooth and projective variety. Its \emph{Hochschild cohomology} is
  \begin{equation}
    \HHHH^\bullet(X)\colonequals\bigoplus_{i=0}^{2\dim X}\HHHH^i(X)
  \end{equation}
  for
  \begin{equation}
    \label{equation:hochschild-cohomology-definition}
    \HHHH^i(X)\colonequals\Ext_{X\times X}^i(\Delta_*\mathcal{O}_X,\Delta_*\mathcal{O}_X),
  \end{equation}
  where~$\Delta\colon X\hookrightarrow X\times X$ denotes the diagonal embedding.
\end{definition}

The Hochschild--Kostant--Rosenberg decomposition gives a convenient description of the summands $\HHHH^i(X)$ in terms of \emph{polyvector fields}, and it is obtained via the \emph{Hochschild--Kostant--Rosenberg quasi-isomorphism}~$\LLL\Delta^*\circ\Delta_*\mathcal{O}_X\cong\bigoplus_{i=0}^{\dim X}\Omega_X^i[i]$ considered in \cite{MR2141853,MR1940241,MR2472137}.

\begin{theorem}[Hochschild--Kostant--Rosenberg decomposition]
  \label{theorem:naive-hkr}
  Let~$X$ be a smooth projective variety. Then there exists an isomorphism
  \begin{equation}
    \HHHH^i(X)\cong\bigoplus_{p+q=i}\HH^p(X,\bigwedge^q\tangent_X)
  \end{equation}
  for~$i=0,\ldots,2\dim X$ induced by the Hochschild--Kostant--Rosenberg quasi-isomorphism.
\end{theorem}
Hence as a first approximation (disregarding any algebraic structures present on Hochschild cohomology) determining the Hochschild cohomology of a variety reduces to a question in sheaf cohomology.

\begin{remark}
  There is also the \emph{Hochschild homology} of~$X$, defined as
  \begin{equation}
    \HHHH_\bullet(X)\colonequals\bigoplus_{i=-\dim X}^{\dim X}\HHHH_i(X)
  \end{equation}
  where
  \begin{equation}
    \HHHH_i(X)\colonequals\Ext_{X\times X}^{i+\dim X}(\Delta_*\mathcal{O}_X,\Delta_*\omega_X).
  \end{equation}
  Moreover there is the Hochschild--Kostant--Rosenberg decomposition for Hochschild homology, which now reads
  \begin{equation}
    \HHHH_i(X)\cong\bigoplus_{p-q=i}\HH^q(X,\Omega_X^p)
  \end{equation}
  for~$i=-\dim X,\ldots,\dim X$. Hence the dimension of the Hochschild homology of~$X$ is determined by the Hodge numbers~$\hh^{p,q}=\hh^q(X,\Omega_X^p)$. These numbers admit symmetries under Serre duality and Hodge symmetry, and therefore are often written down in the form of the \emph{Hodge diamond}. In particular for Fano 3-folds the Hodge diamond is of the form
  \begin{equation}
    \label{equation:hodge-diamond-fano-3-fold}
    \begin{array}{*8{>{}c}} {\scriptstyle\HHHH_{-3}(X)} & {\scriptstyle\HHHH_{-2}(X)} & {\scriptstyle\HHHH_{-1}(X)} & {\scriptstyle\HHHH_0(X)} & {\scriptstyle\HHHH_1(X)} & {\scriptstyle\HHHH_2(X)} & {\scriptstyle\HHHH_3(X)} \\ &&& 1 \\ && 0 && 0 \\ & 0 && \rho && 0 \\ 0 && \hh^{1,2} && \hh^{1,2} && 0 \\ & 0 && \rho && 0 \\ && 0 && 0 \\ &&& 1 \end{array}
  \end{equation}
  and it is determined by the invariants from the classification. The dimensions of the Hochschild homology spaces now correspond to different columns in this diamond (as opposed to the rows which describe the dimensions of singular cohomology spaces).
\end{remark}

To mimic this economical description of the Hochschild--Kostant--Rosenberg decomposition of Hochschild homology using the Hodge diamond, the first author introduced the \emph{polyvector parallelogram}. If we denote~$\pv^{p,q}\colonequals\dim_k\HH^p(X,\bigwedge^q\tangent_X)$, then for a 3-fold it is given by
\begin{equation}
  \label{equation:parallelogram}
  \arraycolsep=10pt
  \renewcommand\arraystretch{1.2}
  \begin{array}{*5{>{}c}} {\scriptstyle\HHHH^0(X)} & 1 \\ {\scriptstyle\HHHH^1(X)} & \pv^{1,0} & \pv^{0,1} \\ {\scriptstyle\HHHH^2(X)} & \pv^{2,0} & \pv^{1,1} & \pv^{0,2} \\ {\scriptstyle\HHHH^3(X)} & \pv^{3,0} & \pv^{2,1} & \pv^{1,2} & \pv^{0,3} \\ {\scriptstyle\HHHH^4(X)} & & \pv^{3,1} & \pv^{2,2} & \pv^{1,3} \\ {\scriptstyle\HHHH^5(X)} & && \pv^{3,2} & \pv^{2,3} \\ {\scriptstyle\HHHH^6(X)} & &&& \pv^{3,3} \end{array}
\end{equation}
with an obvious generalisation to other dimensions. There are no symmetries present in the numbers~$\pv^{p,q}$, and the presentation reflects this absence.

\begin{remark}
  Another important difference between the Hodge diamond and the polyvector parallelogram is that the former is constant in families, whilst the latter is not necessarily so. We will explain this for Fano 3-folds in \cref{subsection:general-results}.
\end{remark}

\paragraph{Additional structure}
There is a rich algebraic structure on Hochschild cohomology~$\HHHH^\bullet(X)$, and on the polyvector fields~$\bigoplus_{p+q=\bullet}\HH^p(X,\bigwedge^q\tangent_X)$. Namely there exist:
\begin{itemize}
  \item a graded-commutative product (of degree 0);
  \item a graded Lie bracket (of degree $-1$)
\end{itemize}
which are related via the Poisson identity, yielding the structure of a \emph{Gerstenhaber algebra}.

On Hochschild cohomology this structure can be either induced using a localised version of the Hochschild cochain complex of an algebra \cite{MR1940241,MR1855264}, or the general machinery of Hochschild cohomology for dg~categories \cite{keller-dih}. The product corresponds to the Yoneda product on self-extensions in \eqref{equation:hochschild-cohomology-definition}, whilst the Gerstenhaber bracket~$[-,-]$ does not have a direct sheaf-theoretic interpretation in the definition \eqref{equation:hochschild-cohomology-definition}.

For polyvector fields the product structure is given by the cup product in sheaf cohomology together with the wedge product of polyvector fields, whilst the Lie bracket is given by the Schouten bracket~$[-,-]_{\mathrm{S}}$. In this case the Gerstenhaber algebra structure is even compatible with the bigrading.

The isomorphism used in \cref{theorem:naive-hkr} is \emph{not} compatible with the Gerstenhaber algebra structures on both sides. This was remedied by Kontsevich (see \cite[Claim~8.4]{MR2062626} and \cite[Theorem~5.1]{MR2141853}) for the algebra structure and Calaque--Van den Bergh \cite[Corollary~1.5]{MR2646112} for the full Gerstenhaber algebra structure, by modifying it using the square root of the Todd class. We will denote the isomorphism~$\HHHH^\bullet(X)\cong\bigoplus_{p+q=\bullet}\HH^p(X,\bigwedge^q\tangent_X)$ of graded vector spaces obtained from \cref{theorem:naive-hkr} by~$\mathrm{I}^{\mathrm{HKR}}$.
\begin{theorem}[Kontsevich, Calaque--Van den Bergh]
  We have an isomorphism of Gerstenhaber algebras
  \begin{equation}
    \mathrm{I}^{\mathrm{HKR}}\circ\sqrt{\todd_X}\wedge-\colon\bigoplus_{p+q=\bullet}\HH^p(X,\bigwedge^q\tangent_X)\overset{\cong}{\to}\HHHH^\bullet(X).
  \end{equation}
\end{theorem}
By describing the algebraic structure on polyvector fields we can therefore deduce properties of the algebraic structure on Hochschild cohomology of~$X$.

We can identify certain interesting substructures:
\begin{itemize}
  \item $(\HHHH^1(X),[-,-])$ is a Lie algebra,
  \item $\HHHH^i(X)$ is a representation of~$(\HHHH^1(X),[-,-])$,
  \item the self-bracket~$[\alpha,\alpha]\in\HHHH^3(X)$ for~$\alpha\in\HHHH^2(X)$ measures the obstruction to extending a first-order deformation of the abelian or derived category of coherent sheaves (classified by~$\HHHH^2(X)$, see \cite{MR2183254,MR2238922}) to higher order,
\end{itemize}
whilst on the polyvector fields and using the finer bigrading we have that:
\begin{itemize}
  \item $(\HH^0(X,\tangent_X),[-,-]_{\mathrm{S}})$ is the Lie algebra~$\Lie\Aut(X)$;
  \item $\bigoplus_{p+q=i}\HH^p(X,\bigwedge^q\tangent_X)$ is a bigraded representation of~$\Lie\Aut(X)$;
  \item the self-bracket~$[\beta,\beta]_{\mathrm{S}}\in\HH^2(X,\tangent_X)$ for~$\beta\in\HH^1(X,\tangent_X)$ measures the obstruction to extending a first-order deformation of the variety~$X$ to higher order in the Kodaira--Spencer deformation theory of varieties.
\end{itemize}
For a Fano variety the latter obstruction vanishes as~$\HH^2(X,\tangent_X)=0$ by Kodaira--Akizuki--Nakano vanishing, see also \cref{lemma:kodaira}. By \cite{MR3985696} the Lie algebra~$\Lie\Aut(X)$ is non-trivial in many cases, and it would be interesting (but outside the scope of this article) to describe this aspect of the Gerstenhaber algebra structure.

There is also the self-bracket~$[\pi,\pi]_{\mathrm{S}}\in\HH^0(X,\bigwedge^3\tangent_X)$ for~$\pi\in\HH^0(X,\bigwedge^2\tangent_X)$, which we will now elaborate on. By Kodaira vanishing~$\HH^2(X,\mathcal{O}_X)$ will play no role in this article.

\subsection{Poisson structures}
A \emph{Poisson structure} is a~$k$-bilinear operation~$\{-,-\}\colon\mathcal{O}_X\times\mathcal{O}_X\to\mathcal{O}_X$ satisfying the axioms of a Poisson bracket; in particular it satisfies the Jacobi identity. It can also be encoded \emph{globally} as a section~$\pi\in\HH^0(X,\bigwedge^2\tangent_X)$, using the equality~$\{f,g\}=\langle\dd f\wedge\dd g,\pi\rangle$ obtained from the pairing between vector fields and differential forms. The vanishing of the Schouten bracket
\begin{equation}
  \label{equation:schouten-vanishes-for-poisson}
  [\pi,\pi]_{\mathrm{S}} = 0 \in\HH^0(X,\bigwedge^3\tangent_X)
\end{equation}
encodes the Jacobi identity for the corresponding Poisson structure. We will use the following terminology.
\begin{definition}
  Let~$X$ be a smooth projective variety. A \emph{Poisson structure} on~$X$ is a bivector field~$\pi\in\HH^0(X,\bigwedge^2\tangent_X)$ such that~\eqref{equation:schouten-vanishes-for-poisson} holds. We denote
  \begin{equation}
    \Pois(X)\subseteq\HH^0(X,\bigwedge^2\tangent_X)
  \end{equation}
  the subvariety of Poisson structures.
\end{definition}
In general~$\Pois(X)$ is cut out by homogeneous equations of degree~2, and one can also consider them up to rescaling, so that one is interested in~$\mathbb{P}(\Pois(X))\subseteq\mathbb{P}(\HH^0(X,\bigwedge^2\tangent_X))$. There can be multiple irreducible components, of varying dimension. For an excellent introduction to Poisson structures, one is referred to \cite{MR3765972}. Let us just recall that Poisson structures are important to construct deformation quantisations, or noncommutative deformations, as e.g.~explained in \cite{bondal-mpi}.

The classification of Poisson structures on smooth projective surfaces is done in \cite{MR2129789}, with the vanishing of the Schouten bracket being automatic for dimension reasons. The classification of Poisson structures Fano 3-folds of Picard rank 1 is summarised in \cite[\S9, Table~1]{MR3066408}. We don't need the full classification, let us just mention the following examples.
\begin{example}
  By \cite[\S9, Table~1]{MR3066408} we have that
  \begin{itemize}
    \item for~$\mathbb{P}^3$ there are 6~irreducible components, of varying dimension;
    \item in the family \MM{1}{10} there exists a unique member for which~$\mathbb{P}(\Pois(X))$ is non-empty in~$\mathbb{P}(\HH^0(X,\bigwedge^2\tangent_X))\cong\mathbb{P}^2$, in which case it is a point: the Mukai--Umemura 3-fold~$X^{\mathrm{MU}}$ for which~$\Aut^0(X^{\mathrm{MU}})=\PGL_2$;
    \item in the family \MM{1}{9} we have for all~$X$ that~$\mathbb{P}(\Pois(X))=\emptyset$ inside~$\mathbb{P}(\HH^0(X,\bigwedge^2\tangent_X))=\pt$.
  \end{itemize}
\end{example}
As mentioned in \cite[\S3.4]{MR3765972}, the full classification of Poisson structures on Fano~3-folds of higher Picard rank is still open, and \cref{corollary:no-poisson} gives the first step towards such a classification.

\section{Computing the Hochschild cohomology of Fano 3-folds}
\label{section:computation}
In this section we discuss the aspects of the computation of Hochschild cohomology of Fano 3-folds which are common to all cases. After introducing some general results in \cref{subsection:general-results} we will set up the computation in \cref{subsection:setup} and discuss the two approaches in \cref{subsection:toric,subsection:homogeneous}. For the remaining cases one is referred to \cref{section:by-hand}.

\subsection{General results}
\label{subsection:general-results}
The following lemma is straightforward, but significantly reduces the number of cohomologies one needs to compute for a Fano~3-fold.
\begin{lemma}
  \label{lemma:kodaira}
  Let~$X$ be a Fano 3-fold. Then
  \begin{equation}
    \begin{aligned}
      \HH^\bullet(X,\mathcal{O}_X)&=\HH^0(X,\mathcal{O}_X)\cong k[0] \\
      \HH^\bullet(X,\tangent_X)&=\HH^0(X,\tangent_X)\oplus\HH^1(X,\tangent_X) \\
      \HH^\bullet(X,\bigwedge^2\tangent_X)&=\HH^0(X,\bigwedge^2\tangent_X)\oplus\HH^1(X,\bigwedge^2\tangent_X)\oplus\HH^2(X,\bigwedge^2\tangent_X) \\
      \HH^\bullet(X,\bigwedge^3\tangent_X)&=\HH^0(X,\omega_X^\vee).
    \end{aligned}
  \end{equation}
\end{lemma}

\begin{proof}
  This is immediate from the Kodaira--Akizuki--Nakano vanishing
  \begin{equation}
    \HH^q(X,\mathcal{L}\otimes\Omega_X^p)=0\qquad\forall p+q>\dim X
  \end{equation}
  for an ample line bundle~$\mathcal{L}$, by considering~$(p,\mathcal{L})=(3,\omega_X^\vee),(2,\omega_X^\vee),(1,\omega_X^\vee),(3,\omega_X^\vee\otimes\omega_X^\vee)$ and using the identification~$\bigwedge^i\tangent_X\cong\omega_X^\vee\otimes\Omega_X^{3-i}$.
\end{proof}
In particular, the polyvector parallelogram introduced in \cref{subsection:hochschild-cohomology} has the form
\begin{equation}
  \label{equation:parallelogram-fano-3-fold}
  \arraycolsep=10pt
  \renewcommand\arraystretch{1.2}
  \begin{array}{*5{>{}c}} {\scriptstyle\HHHH^0(X)} & 1 \\ {\scriptstyle\HHHH^1(X)} & 0 & \pv^{0,1} \\ {\scriptstyle\HHHH^2(X)} & 0 & \pv^{1,1} & \pv^{0,2} \\ {\scriptstyle\HHHH^3(X)} & 0 & 0 & \pv^{1,2} & \pv^{0,3} \\ {\scriptstyle\HHHH^4(X)} & & 0 & \pv^{2,2} & 0 \\ {\scriptstyle\HHHH^5(X)} & && 0 & 0 \\ {\scriptstyle\HHHH^6(X)} & &&& 0 \end{array}
\end{equation}

Next we describe the Euler characteristic of the vector bundles appearing in \cref{lemma:kodaira}.
Recall that Hirzebruch--Riemann--Roch for a vector bundle~$\mathcal{E}$ on a 3-fold takes on the following form, where we abbreviate~$\cc_i=\cc_i(\tangent_X)$:
\begin{equation}
  \label{equation:hrr-fano-3-fold}
  \begin{aligned}
    \chi(X,\mathcal{E})
    &=\int_X\chern(\mathcal{E})\todd_X \\
    &=\frac{1}{24}\rk(\mathcal{E})\cc_1\cc_2 + \frac{1}{12}\cc_1(\mathcal{E})\left( \cc_1^2 + \cc_2 \right) + \frac{1}{4}\left( \cc_1(\mathcal{E})^2 - 2\cc_2(\mathcal{E}) \right)\cc_1 \\
    &\qquad + \frac{1}{6}\left( \cc_1(\mathcal{E})^3 - 3\cc_1(\mathcal{E})\cc_2(\mathcal{E}) + 3\cc_3(\mathcal{E}) \right).
  \end{aligned}
\end{equation}
We obtain the following identities, expressing the Euler characteristic of the bundles we are interested in in terms of the usual invariants~$\rho$, $\hh^{1,2}$ and~$\cc_1^3$ in the classification of Fano 3-folds.
\begin{lemma}
  \label{lemma:hrr-computations}
  Let~$X$ be a Fano 3-fold.
  \begin{align}
    \chi(X,\tangent_X)
    &=\frac{1}{2}\cc_1^3 + \rho - 18 - \hh^{1,2};
    \label{equation:hrr-tangent} \\
    \chi(X,\bigwedge^2\tangent_X)
    &=\cc_1^3 - 18 - \rho + \hh^{1,2};
    \label{equation:hrr-bivectors} \\
    \chi(X,\omega_X^\vee)
    &=\frac{1}{2}\cc_1^3+3.
    \label{equation:anticanonical}
  \end{align}
\end{lemma}

\begin{proof}
  By \eqref{equation:hrr-fano-3-fold} for~$\mathcal{O}_X$ and Kodaira vanishing we have that~$\chi(X,\mathcal{O}_X)=\frac{\cc_1\cc_2}{24}=1$, so
  \begin{equation}
    \cc_1\cc_2=24.
  \end{equation}
  And~$\cc_3$ is the topological Euler characteristic, so
  \begin{equation}
    \cc_3=2+2\rho-2\hh^{1,2}.
  \end{equation}
  Hence \eqref{equation:hrr-tangent} and \eqref{equation:anticanonical} follow from \eqref{equation:hrr-fano-3-fold}.

  For \eqref{equation:hrr-bivectors} we use that
  \begin{equation}
    \begin{aligned}
      \chern(\bigwedge^2\tangent_X)
      &=\chern(\Omega_X^1)\chern(\omega_X^\vee) \\
      &=\chern(\tangent_X^\vee)\chern(\omega_X^\vee) \\
      &=\left( 3 - \cc_1 + \frac{1}{2}(\cc_1^2 - 2\cc_2) + \frac{1}{6}(-\cc_1^3 + 3\cc_1\cc_2 - 3\cc_2) \right)\left( 1 + \cc_1 + \frac{1}{2}\cc_1^2 + \frac{1}{6}\cc_1^3 \right) \\
      &=3 + 2\cc_1 + \cc_1^2 - \cc_2 + \frac{1}{3}\cc_1^3 - \frac{1}{2}\cc_1\cc_2 - \frac{1}{2}\cc_3
    \end{aligned}
  \end{equation}
  so that reading off the degree three part of~$\chern(\bigwedge^2\tangent_X)\todd_X$ gives
  \begin{equation}
    \chi(X,\bigwedge^2\tangent_X)=\cc_1^3 - \frac{17}{24}\cc_1\cc_2 - \frac{1}{2}\cc_3
  \end{equation}
  and the identity in \eqref{equation:hrr-bivectors} follows from the observations made in the previous paragraph.
\end{proof}

This observation, together with the classification of infinite automorphism groups of Fano~3-folds (see \cite[Theorem~1.1.2]{MR3776469} for Picard rank~1, and \cite[Theorem~1.2]{MR3985696} for Picard rank~$\geq2$), makes it straightforward to determine~$\hh^0(X,\tangent_X)$ and~$\hh^1(X,\tangent_X)$.
\begin{proposition}
  Let~$X$ be a Fano 3-fold. We have that
  \begin{equation}
    \begin{aligned}
      \hh^0(X,\tangent_X)&=\dim\Aut^0(X); \\
      \hh^1(X,\tangent_X)&=-\left( \frac{1}{2}\cc_1^3 + \rho - 18 - \hh^{1,2} \right)-\dim\Aut(X).
    \end{aligned}
  \end{equation}
\end{proposition}
The computation of~$\Aut^0(X)$ can be found in \cite[Table~1]{MR3985696}. It is important to note that the dimension of~$\Aut(X)$ can \emph{vary in families}.

For~$\bigwedge^2\tangent_X$ we need to determine 3 possibly non-zero cohomologies, and none is known a priori. Some cases are easy (e.g.~for toric Fano 3-folds Bott--Steenbrink--Danilov vanishing, see e.g.~\cite[Theorem~2.4]{MR1916637}, yields that~$\HH^{\geq 1}(X,\bigwedge^i\tangent_X)=0$) but others take more effort.

\subsection{Setting up the computation}
\label{subsection:setup}
As discussed in the previous section, it suffices to compute the cohomology of~$\bigwedge^2\tangent_X$ to fully determine the Hochschild cohomology of a Fano 3-fold. By \cref{lemma:kodaira} we know that its cohomology is concentrated in degrees~$0,1,2$.

To perform this computation we will use suitable descriptions of Fano 3-folds~$X$ inside key varieties~$F$ provided in \cite{2009.13382v1,MR3470714}. A key variety will be either a product of Grassmannians or a toric variety. In the former case~$X$ is given as the zero locus of a general global section of a homogeneous vector bundle~$\mathcal{E}$ on~$F$. In the latter case~$X$ is given as an intersection of divisors inside a possibly singular~$F$. It turns out that this second description involves non-Cartier divisors only for~\MM{2}{1} and \MM{2}{3}: this will lead us to deal with these two cases separately in \cref{section:by-hand}.

The two methods outlined in this section allow for a near uniform treatment using computer algebra methods. We implemented them using Macaulay2 \cite{M2} and Magma \cite{MR1484478}; our code is publicly available at \cite{code} and can be used to check our computations. As it turns out, this automated treatment leaves the cohomology of~$\bigwedge^2\tangent$ underdetermined for only 5~Fano 3-folds, which require additional computations by hand (2 of which straightforward). These cases will be treated in \cref{section:by-hand}.

\begin{remark}
  For many deformation families of Fano 3-folds one can of course envision alternative methods, e.g.~using descriptions as a blowup, double cover or product. We will not discuss the details for these alternative methods as they do not allow for an automated approach. One potential benefit (for certain applications) of these methods could be that they give a more intrinsic description of the cohomology. Let us just point out that they are used for 5 explicit instances in~\cref{section:by-hand}. \end{remark}

\paragraph{Setup and notation}
Let us introduce some notation, which is also the notation we use in (the documentation of) the ancillary code. Let~$X$ be a Fano 3-fold (not of type \MM{2}{1} or \MM{2}{3}), defined by the vanishing of a global section of a vector bundle~$\mathcal{E}$ inside a key variety~$F$ with $\codim_F X = \rank \mathcal{E}$. By \cref{theorem:fanosearch-description,theorem:homogeneous-description} the key variety can be chosen as either a product of Grassmannians or a (possibly singular) toric variety. We wish to compute the cohomology of
\begin{equation}
  \bigwedge^2\tangent_X\cong\Omega_X^1\otimes\omega_X^\vee.
\end{equation}
We will do this by using the conormal sequence, using that the ideal sheaf~$\mathcal{I}$ cutting out~$X$ gives~$(\mathcal{I}/\mathcal{I}^2)|_X\cong\mathcal{E}^\vee|_X$. Since $X$ is smooth and locally complete intersection within~$F$, one has that $X \subset F^{\text{sm}}$, hence $\Omega_F^1|_X$ is locally free. From \cite[Tags 06AA and 0B3P]{stacks-project} it follows that the conormal sequence
\begin{equation}
  \label{equation:conormal}
  0\to\mathcal{E}^\vee|_X\to\Omega_F^1|_X\to\Omega_X^1\to 0
\end{equation}
is an exact sequence of vector bundles on $X$. We will twist this sequence by the anticanonical bundle~$\omega_X^\vee\cong\omega_F^\vee|_X\otimes\det\mathcal{E}|_X$.
We are interested in computing the cohomologies of the last term of
\begin{equation}
  \label{equation:conormal-twisted}
  0\to(\mathcal{E}^\vee \otimes \omega_F^\vee\otimes\det\mathcal{E})|_X \to(\Omega_F^1 \otimes \omega_F^\vee\otimes\det\mathcal{E})|_X\to\Omega_X^1 \otimes (\omega_F^\vee\otimes\det\mathcal{E})|_X\to 0.
\end{equation}
The first two terms can be resolved by suitable twists of the Koszul complex
\begin{equation}
  \label{equation:koszul}
  0\to\det\mathcal{E}^\vee\to\bigwedge^{\rk\mathcal{E}-1}\mathcal{E}^\vee\to\ldots\to\mathcal{E}^\vee\to\mathcal{O}_F\to\mathcal{O}_X\to 0.
\end{equation}
The whole point of this reduction is that the tensor product of~$\bigwedge^i\mathcal{E}^\vee$ with either of the first two bundles from \eqref{equation:conormal-twisted} can now be expressed in terms of vector bundles on~$F$ for which good computational methods exist:
\begin{itemize}
  \item for toric varieties we can use the work of Eisenbud--Musta\c{t}\u{a}--Stillman \cite{MR1769656}, as implemented in \cite{NormalToricVarieties}, even when the cotangent sheaf is not locally free;
  \item for homogeneous varieties we can use the Borel--Weil--Bott theorem.
\end{itemize}

\subsection{Complete intersections in toric varieties}
\label{subsection:toric}
The majority of the cases will be covered by this method. The starting point is the following theorem, which follows from the case-by-case analysis performed in \cite{MR3470714} for Picard ranks~$2,\ldots,5$, whilst for Picard ranks~$1, 6,\ldots,10$ it follows from the description using weighted projective spaces and del Pezzo surfaces.
\begin{theorem}[Coates--Corti--Galkin--Kasprzyk]
  \label{theorem:fanosearch-description}
  Let~$X$ be a Fano 3-fold. Assume its deformation family is not of type
  \begin{description}
    \item[Picard rank 1] \MM{1}{5}, \MM{1}{6}, \MM{1}{7}, \MM{1}{8}, \MM{1}{9}, \MM{1}{10}, \MM{1}{15};
    \item[Picard rank 2] \MM{2}{14}, \MM{2}{17}, \MM{2}{20}, \MM{2}{21}, \MM{2}{22}, \MM{2}{26}.
  \end{description}
  Then~$X$ has a description as a complete intersection of codimension at most~3\ in a projective toric variety~$F$. Moreover we have that
  \begin{itemize}
    \item $F$ is singular if the deformation family of~$X$ is of type
      \begin{description}
        \item[Picard rank 1] \MM{1}{1}, \MM{1}{11}, \MM{1}{12};
        \item[Picard rank 2] \MM{2}{1}, \MM{2}{2}, \MM{2}{3}, \MM{2}{8};
        \item[Picard rank 3] \MM{3}{1}, \MM{3}{4}, \MM{3}{14}, \MM{3}{16};
        \item[Picard rank 4] \MM{4}{5};
        \item[Picard rank 5] \MM{5}{1};
        \item[Picard rank 9] \MM{9}{1};
        \item[Picard rank 10] \MM{10}{1};
      \end{description}
    \item $X$ is the intersection of Cartier divisors if its deformation family is \emph{not} of type
      \begin{description}
        \item[Picard rank 2] \MM{2}{1}, \MM{2}{3}.
      \end{description}
  \end{itemize}
\end{theorem}
So 90 (resp.~92) out of~105~deformation families admit a description in terms of a toric variety~$F$ and a vector bundle~$\mathcal{E}$ (resp.~reflexive sheaf) so that we can use the combination of the Koszul sequence and the conormal sequence. We will restrict ourselves to the case where~$\mathcal{E}$ is a vector bundle, and we will deal with the ~2~remaining cases \MM{2}{1}, \MM{2}{3} using birational methods in \cref{section:by-hand}. We remark that it is certainly possible to find suitable models for them as complete intersections of Cartier divisors in different toric varieties, but we did not manage to fully determine the cohomology of $\bigwedge^2\tangent_X$ in this way.

In \cref{table:toric-description-overview} we give an overview of the codimension of~$X$ in~$F$, and whether the computational methods can give a fully determined answer for the cohomology of $\bigwedge^2\tangent_X$.
\begin{itemize}
  \item The case \MM{1}{1} can be easily determined from the toric computation together with Kodaira vanishing, see \cref{proposition:1-1}.
\item The case \MM{4}{13} can be computed using the description as a blowup, see \cref{proposition:4-13}.
  \item The case \MM{10}{1} readily follows from applying the K\"unneth formula to~$\mathbb{P}^1\times\mathrm{dP}_8$ (\cref{proposition:10-1}). For \MM{9}{1} a similar argument using $\mathbb{P}^1\times\mathrm{dP}_7$ holds, but we chose to use its description as a homogeneous zero locus, see \cref{table:homogeneous-descriptions}.
\end{itemize}

\begin{table}
  \centering
  \begin{tabular}{ccc}
    \toprule
    type & number of deformation families & underdetermined cases \\\midrule
    toric Fano 3-fold & 18 & none \\
    toric hypersurface & 53 & \MM{1}{1}, \MM{4}{13}, \MM{9}{1} and \MM{10}{1} \\
    toric codimension 2 & 15 & none \\
    toric codimension 3 & 4 & none \\
    \midrule
    total & 90 \\
    \bottomrule
  \end{tabular}
  \caption{Overview of Fano 3-folds with a toric description as a complete intersection of Cartier divisors}
  \label{table:toric-description-overview}
\end{table}

\begin{remark}
  The description in \cite{MR3470714} describes~$F$ as the GIT quotient of an affine space by a torus. To compute cohomology of coherent sheaves on the toric variety~$F$ we need to translate this description to a toric fan, and describe the divisors cutting out~$X$ in this language. See \cite[\S C]{MR3470714} for some background.
\end{remark}

\paragraph{An example: \MM{2}{8}}
We now describe an example of a toric complete intersection, and the different steps in the computation. We will consider the deformation family \MM{2}{8}, whose Mori--Mukai description is given by
\begin{enumerate}
  \item a double cover of $\Bl_p\mathbb{P}^3$ with anticanonical branch locus~$B$ such that~$B\cap E$ is smooth,
  \item a double cover of $\Bl_p\mathbb{P}^3$ with anticanonical branch locus~$B$ such that~$B\cap E$ is singular but reduced,
\end{enumerate}
where~$E$ denotes the exceptional divisor of the blowup~$\Bl_p\mathbb{P}^3\to\mathbb{P}^3$, and the second is a specialisation of the first.
\begin{proposition}
  \label{proposition:2-8}
  Let~$X$ be a Fano 3-fold in the deformation family~\MM{2}{8}. Then we have that~$\hh^i(X,\bigwedge^2\tangent_X)=3,1,1$ for~$i=0,1,2$.
\end{proposition}
By \cite[\S25]{MR3470714} the GIT description of the toric key variety~$F$ for~$X$ is given by the weights
\begin{center}
  \begin{tabular}{c|cccccc}
    $L$ & $1$ & $1$ & $1$ & $-1$ & $0$ & $1$ \\
    $M$ & $0$ & $0$ & $0$ & $1$ & $1$ & $1$
  \end{tabular}
\end{center}
so that the nef cone of~$F$ is spanned by~$L$ and~$L+M$. Then~$X$ is a divisor in the linear system~$|2L+2M|$. Translating this to a description using the set of rays~$R$ and the set of cones~$C$ gives
\begin{equation}
  \begin{aligned}
    R &= \{(-1,-1,-1,1), (0,0,0,1), (0,0,1,0), (0,1,0,0), (1,0,0,0), (1,1,1,-2)\}; \\
    C &= \{(1,2,3,4,5), (0,2,4,5), (0,1,3,4), (0,1,2,4), (0,1,2,3), (0,3,4,5), (0,2,3,5)\}.
  \end{aligned}
\end{equation}

\begin{proof}[Proof of \cref{proposition:2-8}]
  We want to compute the cohomology of the first two terms in the sequence \eqref{equation:conormal} twisted by~$\omega_F^\vee(-2L-2M)$ (which is~$\omega_X^\vee$ before adjunction), so by the Koszul sequence we want to compute the cohomology of the first two terms in the sequences
  \begin{equation}
    0\to\omega_F^\vee(-6L-6M)\to\omega_F^\vee(-4L-4M)\to\omega_F^\vee(-4L-4M)|_X\to 0
  \end{equation}
  and
  \begin{equation}
    0\to\Omega_F^1\otimes\omega_F^\vee(-4L-4M)\to\Omega_F^1\otimes\omega_F^\vee(-2L-2M)\to\Omega_F^1\otimes\omega_F^\vee(-2L-2M)|_X\to 0.
  \end{equation}
  One computes that
  \begin{equation}
    \begin{aligned}
      \hh^i(F,\omega_F^\vee(-6L-6M))&=0,0,0,0,1 \\
      \hh^i(F,\omega_F^\vee(-4L-4M))&=0,0,0,0,0 \\
      \hh^i(F,\Omega_F^1\otimes\omega_F^\vee(-4L-4M))&=0,0,1,0,0 \\
      \hh^i(F,\Omega_F^1\otimes\omega_F^\vee(-2L-2M))&=3,0,0,0,0
    \end{aligned}
  \end{equation}
  for~$i=0,\ldots,4$, which implies the statement after a diagram chase.
\end{proof}

\begin{remark}
  The homogeneous description from \cite{2009.13382v1} involves a vector bundle on~$\mathbb{P}^2\times\mathbb{P}^3\times\mathbb{P}^{12}$ which is not completely reducible, making the description as a toric complete intersection much more economical.
\end{remark}

\subsection{Zero loci of sections of homogeneous vector bundles}
\label{subsection:homogeneous}

By specialising \cite[Theorems~1.1 and~1.2]{2009.13382v1} to the remaining cases we can paraphrase the main result from op.~cit. for the relevant subset as follows.
\begin{theorem}
  \label{theorem:homogeneous-description}
  Let~$X$ be a Fano 3-fold. Assume its deformation type is not covered by \cref{theorem:fanosearch-description}, or is \MM{9}{10}. Then~$X$ is the zero locus of a (general) global section of a completely reducible homogeneous vector bundle on a product of Grassmannians. The description is given in \cref{table:homogeneous-descriptions}.
\end{theorem}

As explained in op.~cit.~this in fact holds for 85 out of 105 deformation families of Fano 3-folds, but we need it only for the 14~cases specified in \cref{theorem:homogeneous-description} which are listed in \cref{table:homogeneous-descriptions}. See also \cref{remark:canapplyhomogeneous}.

\begin{table}
  \centering
  \begin{tabular}{cccc}
    \toprule
    $X$                 & $F$                                                & $\mathcal{E}$                                                                                        & codimension \\
    \midrule
    \MM{1}{5}           & $\Gr(2,5)$                                         & $\mathcal{O}(2)\oplus\mathcal{O}(1)^{\oplus2}$                                                       & 3 \\
    \MM{1}{6}           & $\Gr(2,5)$                                         & $\mathcal{U}^\vee(1)\oplus\mathcal{O}(1)$                                                            & 3 \\
    \MM{1}{7}           & $\Gr(2,6)$                                         & $\mathcal{O}(1)^{\oplus5}$                                                                           & 5 \\
    \MM{1}{8}           & $\Gr(3,6)$                                         & $\bigwedge^2\mathcal{U}^\vee\oplus\mathcal{O}(1)^{\oplus3}$                                          & 6 \\
    \MM{1}{9}           & $\Gr(2,7)$                                         & $\mathcal{Q}^\vee(1)\oplus\mathcal{O}(1)^{\oplus2}$                                                  & 7 \\
    \MM{1}{10}          & $\Gr(3,7)$                                         & $(\bigwedge^2\mathcal{U}^\vee)^{\oplus3}$                                                            & 9 \\
    \MM{1}{15}          & $\Gr(2,5)$                                         & $\mathcal{O}(1)^{\oplus3}$                                                                           & 3 \\
    \addlinespace
    \MM{2}{14}          & $\Gr(2,5)\times\mathbb{P}^1$                       & $\mathcal{O}(1,0)^{\oplus3}\oplus\mathcal{O}(1,1)$                                                   & 4 \\
    \MM{2}{17}          & $\Gr(2,4)\times\mathbb{P}^3$                       & $\mathcal{U}^\vee\boxtimes\mathcal{O}_{\mathbb{P}^3}(1)\oplus\mathcal{O}(1,1)\oplus\mathcal{O}(1,0)$ & 4 \\
    \MM{2}{20}          & $\Gr(2,4)\times\mathbb{P}^2$                       & $\mathcal{U}^\vee\boxtimes\mathcal{O}_{\mathbb{P}^2}(1)\oplus\mathcal{O}(1,0)^{\oplus3}$             & 3 \\
    \MM{2}{21}          & $\Gr(2,4)\times\mathbb{P}^4$                       & $(\mathcal{U}^\vee\boxtimes\mathcal{O}_{\mathbb{P}^4}(1))^{\oplus 2}\oplus\mathcal{O}(1,0)$          & 5 \\
    \MM{2}{22}          & $\Gr(2,5)\times\mathbb{P}^3$                       & $\mathcal{Q}(1)\boxtimes\mathcal{O}_{\mathbb{P}^3}\oplus\mathcal{O}(0,1)^{\oplus3}$                  & 6 \\
    \MM{2}{26}          & $\Gr(2,4)\times\Gr(2,5)$                           & $\mathcal{Q}\boxtimes\mathcal{U}^\vee\oplus\mathcal{O}(1,0)\oplus\mathcal{O}(0,1)^{\oplus2}$         & 7 \\
    \addlinespace
    \MM{9}{1}           & $\mathbb{P}^1\times\mathbb{P}^2\times\mathbb{P}^1$ & $\mathcal{O}(2,2,0)$                                                                                 & 1 \\
    \bottomrule
  \end{tabular}
  \caption{Description as homogeneous zero loci}
  \label{table:homogeneous-descriptions}
\end{table}

\paragraph{An example: \MM{2}{17}}
In this subsection we exhibit a detailed example of the computation where the Fano 3-fold does not admit (at least a priori) a model as a complete intersection in a suitable toric variety. We will use the description given in \cite[\S34]{MR3470714} and \cite[Table~1]{2009.13382v1} and recalled in \cref{table:homogeneous-descriptions}. The deformation family \MM{2}{17}, originally described by Mori and Mukai as the blow up of the quadric 3-fold in an elliptic quintic, is realised as the zero locus~$\mathscr{Z}(\mathcal{E}) \subset F \colonequals \Gr(2,4) \times \mathbb{P}^3$ where
\begin{equation}
  \mathcal{E}\colonequals\mathcal{U}^{\vee}_{\Gr(2,4)}(0,1) \oplus \mathcal{O}(1,1) \oplus \mathcal{O}(1,0)
\end{equation}
is a rank 4 vector bundle.

\begin{proposition}
  Let~$X$ be a Fano~3-fold in the deformation family~\MM{2}{17}. Then we have that~$\hh^i(X,\bigwedge^2\tangent_X)=5,0,0$ for~$i=0,1,2$.
\end{proposition}

\begin{proof}
  We will follow the strategy used in \cite[\S3.3]{2009.13382v1} and summarised in \cref{subsection:setup}. We need to compute the cohomologies of the first two terms in \eqref{equation:conormal-twisted}, which are resolved by exact complexes of locally free sheaves, namely the twists of the Koszul complex \eqref{equation:koszul} by $\mathcal{E}^\vee \otimes \omega_F^\vee\otimes\det\mathcal{E}$ and $\Omega_F^1 \otimes \omega_F^\vee\otimes\det\mathcal{E}$. In this case we have
  \begin{equation}
    \begin{aligned}
      \Omega_F^1&= \mathcal{U}_{\Gr(2,4)} \otimes \mathcal{Q}^{\vee}_{\Gr(2,4)} \oplus \mathcal{Q}^{\vee}_{\mathbb{P}^3}(-1), \\
      \omega_F&= \mathcal{O}_F(-4,-4).
    \end{aligned}
  \end{equation}
  Each term of the locally free resolutions is a completely reducible vector bundle on $F$, and we can use the Borel--Weil--Bott theorem to compute its cohomology. It turns out that there are only 2 non-zero cohomology groups for the first two terms of \eqref{equation:conormal-twisted} tensored with~$\bigwedge^i\mathcal{E}^\vee$ before restriction, for with~$i=0,\ldots,4=\rk\mathcal{E}$, namely
\begin{equation}
    \begin{aligned}
      \hh^0(F,\mathcal{E}^\vee \otimes \omega_F^\vee\otimes\det\mathcal{E})&=9 \\
      \hh^1(F,(\Omega_F^1 \otimes \omega_F^\vee\otimes\det\mathcal{E})\otimes\mathcal{E}^\vee)&=14.
    \end{aligned}
  \end{equation}
  From this we get that the only non-zero cohomologies of the first two terms of \eqref{equation:conormal-twisted} are
  \begin{equation}
    \begin{aligned}
      \hh^0(X,(\mathcal{E}^\vee \otimes \omega_F^\vee\otimes\det\mathcal{E})|_X)&=9 \\
      \hh^0(X,(\Omega_F^1 \otimes \omega_F^\vee\otimes\det\mathcal{E})|_X)&=14
    \end{aligned}
  \end{equation}
  and the statement follows.
\end{proof}

\begin{remark}
  \label{remark:canapplyhomogeneous}
  It is possible to apply the description as a zero locus in a homogeneous variety to \emph{all} Fano 3-folds, but for the purpose of this paper we only do this for the 14 cases listed in \cref{table:homogeneous-descriptions}.

  The benefit of the toric description is that the codimension is usually (much) lower, making the computation faster and having less places where indeterminacies can occur. E.g.~for~\MM{3}{9} the description from \cite{2009.13382v1} has codimension~25, which requires a lengthy Koszul computation. Another complication in the computations in the homogeneous setting is that for some Fano 3-folds the homogeneous bundle used in the description is not completely reducible.
\end{remark}

\section{Underdetermined cases}
\label{section:by-hand}
As discussed above there are just a few cases which require additional computations. These are~\MM{1}{1}, \MM{2}{1}, \MM{2}{3}, \MM{4}{13}, and \MM{10}{1}. We will collect the details for them here. For \MM{2}{1}, \MM{2}{3} and \MM{4}{13} they are somewhat tedious cohomology computations using the birational description. If it were not for the efficiency of the toric and homogeneous computations the majority of the Fano 3-folds would have to be tackled in this way.

The first one is straightforward.
\begin{proposition}
  \label{proposition:1-1}
  Let~$X$ be a Fano~3-fold in the deformation family~\MM{1}{1}. Then we have that~$\hh^i(X,\bigwedge^2\tangent_X)=0,0,35$ for~$i=0,1,2$.
\end{proposition}

\begin{proof}
  In this case~$X$ is a sextic hypersurface in the weighted projective space~$\mathbb{P}(1^4,3)$. Using the method from \cref{subsection:toric} on this description for the toric variety~$\mathbb{P}(1^4,3)$ we immediately obtain that~$\hh^i(X,\bigwedge^2\tangent_X)=0,0,35+a,a$ for~$i=0,1,2,3$ for some~$a\geq 0$. But by \cref{lemma:kodaira} we have that~$\hh^3(X,\bigwedge^2\tangent_X)=0$, so~$a=0$.

Alternatively, one can use that this is a double cover~$f\colon X\to Y$ of~$\mathbb{P}^3=Y$ with a smooth sextic surface~$S$ as branch locus. To do so, recall the short exact sequence
  \begin{equation}
    0\to\Omega_Y^1\to\Omega_Y^1(\log S)\to\mathcal{O}_S\to 0
  \end{equation}
  and the isomorphism
  \begin{equation}
    f_*(\Omega_X^1)\cong\Omega_Y^1\oplus\Omega_Y^1(\log S)\otimes\mathcal{O}_Y(-3)
  \end{equation}
  from \cite[\S2.3 and Lemma~3.16(d)]{MR1193913}, together with the isomorphism
  \begin{equation}
    \omega_X^\vee\cong f^*(\omega_Y^\vee\otimes\mathcal{O}_Y(3)).
  \end{equation}
  This allows one to compute~$\HH^i(X,\Omega_X^1\otimes\omega_X^\vee)$, and the only non-vanishing cohomology lives in degree~2 and is isomorphic to
  \begin{equation}
    \HH^2(Y,\mathcal{O}_S(-2))\cong\HH^3(Y,\mathcal{O}_Y(-8))\cong\HH^0(Y,\mathcal{O}_Y(4))
  \end{equation}
  which is~35-dimensional.
\end{proof}

For the next three cases we will resort to the \emph{birational} description by Mori--Mukai. Namely we will consider the situation of a Fano 3-fold~$X$ which is the blowup of a complete intersection curve~$Z$ inside another Fano~3-fold~$Y$. Let us denote the blowup square as
\begin{equation}
  \label{equation:blowup-setup}
  \begin{tikzcd}
    E \arrow[r, hook, "j"] \arrow[d, "p", swap] & X \arrow[d, "f"] \\
    Z \arrow[r, hook, "i"] & Y
  \end{tikzcd}
\end{equation}
and consider the short exact sequence
\begin{equation}
  0\to f^*(\Omega_Y^1)\to\Omega_X^1\to j_*(\Omega_{E/Z}^1)\to 0.
\end{equation}
We wish to compute the cohomology of the middle term after twisting by~$\omega_X^\vee=f^*(\omega_Y^\vee)\otimes\mathcal{O}_X(-E)$, i.e.~we will consider the short exact sequence
\begin{equation}
  \label{equation:blowup-twisted-ses}
  0\to f^*(\Omega_Y^1\otimes\omega_Y^\vee)\otimes\mathcal{O}_X(-E)\to\bigwedge^2\tangent_X\to j_*(\Omega_{E/Z}^1)\otimes\omega_X^\vee\to 0.
\end{equation}

\begin{lemma}
  \label{lemma:blowup-vanishing}
  With the setup from \eqref{equation:blowup-setup} we have that
  \begin{equation}
    \HH^\bullet(X,j_*(\Omega_{E/Z}^1)\otimes\omega_X^\vee)=0.
  \end{equation}
\end{lemma}

\begin{proof}
  Since~$Z$ has codimension~2, we have that~$\Omega_{E/Z}^1\cong\mathcal{O}_p(-2)\otimes p^*(\mathcal{L})$ for some line bundle~$\mathcal{L}$ on~$Z$, whilst~$\omega_X|_E=\mathcal{O}_p(-1)$. This implies that~$j_*(\Omega_{E/Z}^1)\otimes\omega_X^\vee\cong j_*(\mathcal{O}_p(-1)\otimes p^*\mathcal{L})$. But then the vanishing of~$\RRR p_*\mathcal{O}_p(-1)$ ensures that
  \begin{equation}
    \HH^\bullet(X,j_*(\Omega_{E/Z}^1)\otimes\omega_X^\vee)
    \cong\HH^\bullet(E,\mathcal{O}_p(-1)\otimes p^*\mathcal{L})
    \cong\HH^\bullet(Z,\RRR p_*\mathcal{O}_p(-1)\otimes\mathcal{L})
    =0,
  \end{equation}
which is what we wanted to show.
\end{proof}

\begin{corollary}
  \label{corollary:blowup-identification}
  With the setup from \eqref{equation:blowup-setup} we have that
  \begin{equation}
    \HH^\bullet(X,\bigwedge^2\tangent_X)\cong\HH^\bullet(Y,\Omega_Y^1\otimes\omega_Y^\vee\otimes\mathcal{I}_Z).
  \end{equation}
\end{corollary}

\begin{proof}
  This follows from \eqref{equation:blowup-twisted-ses}, the vanishing in \cref{lemma:blowup-vanishing}, the isomorphism~$\RRR f_*(\mathcal{O}_X(-E))\cong\mathcal{I}_Z$, and adjunction.
\end{proof}

We now consider the underdetermined cases \MM{2}{1} and \MM{2}{3}. In this case the methods of \cref{subsection:toric} don't necessarily apply: both are described as a codimension-2 complete intersection in a singular toric projective variety, and in both cases one of the divisors is not Cartier. Therefore we cannot ensure that the computational (underdetermined) answer is correct\footnote{The underdetermined result from the toric computation is nevertheless consistent with the final answer, so a more detailed analysis of the toric description might be valid. But the indeterminacy needs to be dealt with by alternative methods in any case.}, so we will combine \cref{corollary:blowup-identification} with Dolgachev's computation of sheaf cohomology on weighted projective spaces. We will repeatedly make use of Dolgachev's formulae for twisted Hodge numbers on weighted projective spaces, cf.~\cite[\S2.3.2--2.3.5]{MR0704986}.

\begin{proposition}
  \label{proposition:2-1}
  Let~$X$ be a Fano 3-fold in the deformation family~\MM{2}{1}. Then we have that~$\hh^i(X,\bigwedge^2\tangent_X)=1,2,7$ for~$i=0,1,2$.
\end{proposition}

\begin{proof}
  Such a variety is described as the blowup of a Fano 3-fold~$Y$ of deformation type~\MM{1}{11} in an elliptic curve obtained as complete intersection of two half-anticanonical divisors, and~$Y$ is given as a sextic hypersurface in the weighted projective space~$\mathbb{P}\colonequals\mathbb{P}(1^3,2,3)$.

  Alternatively, $X$~is a~$(1,1)$-section of~$F\colonequals Y\times \mathbb{P}^1$. We will first use this description to determine~$\hh^0$ and~$\hh^2-\hh^1$, and then use the blowup description to determine~$\hh^2$ (and~$\hh^1-\hh^0$), which determines everything.

  \paragraph{First step: weighted projective space computation}
  Consider the conormal sequence for the inclusion~$X\hookrightarrow F$ twisted by the anticanonical line bundle, which by the adjunction formula is
  \begin{equation}
    \omega_X^\vee\cong(\mathcal{O}_F(2,2)\otimes\mathcal{O}_F(-X))|_X\cong\mathcal{O}_F(1,1)|_X,
  \end{equation}
  which yields
  \begin{equation}
    \label{equation:2-1-conormal}
    0 \to \mathcal{O}_X \to \Omega^1_F(1,1)|_X \to \Omega^1_X(1,1) \to 0.
  \end{equation}
  We claim that
  \begin{equation}
    \label{equation:2-1-weighted-claim}
    \left\{
    \begin{aligned}
      \hh^0(X,\Omega^1_X(1,1))&=1 \\
      \hh^2(X,\Omega^1_X(1,1))-\hh^1(X,\Omega^1_X(1,1))&=5.
    \end{aligned}
    \right.
  \end{equation}
  This claim follows from~$\hh^0(F,\Omega^1_F(1,1)|_X)=2$ and~$\hh^2(F,\Omega^1_F(1,1)|_X)-\hh^1(F,\Omega^1_F(1,1)|_X)=5$ together with \eqref{equation:2-1-conormal}, so let us prove this. Consider the twisted Koszul sequence
  \begin{equation}
    \label{equation:2-1-koszul}
    0 \to \Omega^1_F \to \Omega^1_F(1,1) \to \Omega^1_F(1,1)|_X \to 0.
  \end{equation}
  As~$\Omega^1_F$ is the direct sum of the pullbacks of the cotangent bundles of~$Y$ and~$\mathbb{P}^1$, one readily gets from the Hodge numbers of $\mathbb{P}^1$ and \MM{1}{11} that~$\hh^i(F,\Omega^1_F)=0,2,21,0,0$ for~$i=0,1,2,3,4$. Thus it remains to compute the cohomologies of the middle term of \eqref{equation:2-1-koszul}. To do that, we apply the K\"unneth formula to
  \begin{equation}
    \Omega^1_F(1,1) = \Omega^1_Y(1) \boxtimes \mathcal{O}_{\mathbb{P}^1}(1) \oplus \mathcal{O}_Y(1) \boxtimes \mathcal{O}_{\mathbb{P}^1}(-1);
  \end{equation}
  the second term is acyclic since $\mathcal{O}_{\mathbb{P}^1}(-1)$ is, whilst for the first one we get
  \begin{equation}
    \label{equation:2-1-doubling}
    \hh^j(F,\Omega^1_F(1,1)) = 2\hh^j(Y,\Omega^1_Y(1)).
  \end{equation}
  To compute the latter, we consider the twisted conormal sequence for~$Y$, seen as a sextic hypersurface in~$\mathbb{P}$:
  \begin{equation}
    0 \to \mathcal{O}_Y(-5) \to \Omega^1_{\mathbb{P}}(1)|_Y \to \Omega^1_Y(1) \to 0.
  \end{equation}
  By \cite[\S2.3.2--2.3.5]{MR0704986}
the first term has only one non-vanishing cohomology~$\hh^3(Y,\mathcal{O}_Y(-5)) = 14$. Thus, we have
  \begin{equation}
    \label{equation:2-1-conormal-Y}
    \begin{aligned}
      \hh^0(Y,\Omega^1_Y(1))&=\hh^0(Y,\Omega^1_{\mathbb{P}}(1)|_Y) \\
      \hh^1(Y,\Omega^1_Y(1))&=\hh^1(\mathbb{P},\Omega^1_{\mathbb{P}}(1)|_Y) \\
      \hh^2(Y,\Omega^1_Y(1))-\hh^3(Y,\Omega^1_Y(1))&=14-\hh^3(\mathbb{P},\Omega^1_{\mathbb{P}}(1)|_Y)+\hh^2(\mathbb{P},\Omega^1_{\mathbb{P}}(1)|_Y).
    \end{aligned}
  \end{equation}
  The cohomologies of $\Omega^1_{\mathbb{P}}(1)|_Y$ can be obtained as usual by means of a twisted Koszul complex
  \begin{equation}
    0 \to \Omega^1_{\mathbb{P}}(-5) \to \Omega^1_{\mathbb{P}}(1) \to \Omega^1_{\mathbb{P}}(1)|_Y \to 0.
  \end{equation}
  Now \cite[\S2.3.2--2.3.5]{MR0704986}
yields the cohomologies for the first two terms, whence we deduce that
  \begin{equation}
    \begin{aligned}
      \hh^i(\mathbb{P},\Omega^1_{\mathbb{P}}(1)|_Y)&=0 \quad \mbox{for\ } i=0,1,2 \\
      \hh^3(\mathbb{P},\Omega^1_{\mathbb{P}}(1)|_Y)&=1.
    \end{aligned}
  \end{equation}
  We plug the last equalities into \eqref{equation:2-1-conormal-Y} and get
  \begin{equation}
    \begin{alignedat}{2}
      \hh^0(Y,\Omega^1_Y(1))&=0, & \qquad \hh^1(Y,\Omega^1_Y(1)) &= 0, \\
      \hh^2(Y,\Omega^1_Y(1))&=\alpha + 13, & \qquad\hh^3(Y,\Omega^1_Y(1))&= \alpha
    \end{alignedat}
  \end{equation}
  for some integer $\alpha\geq 0$. By \eqref{equation:2-1-doubling} and \eqref{equation:2-1-koszul} we get
  \begin{equation}
    \begin{alignedat}{2}
      \hh^0(F,\Omega^1_F(1,1)|_X)&=2, & \qquad\hh^1(F,\Omega^1_F(1,1)|_X)&= \beta, \\
      \hh^2(F,\Omega^1_F(1,1)|_X)&=\beta + 2\alpha + (2\cdot13-21), & \qquad\hh^3(F,\Omega^1_F(1,1)|_X)&=2\alpha
    \end{alignedat}
  \end{equation}
  for some integers $\alpha,\beta\geq0$. But by \eqref{equation:2-1-conormal} and \cref{lemma:kodaira} we get~$2\alpha=\hh^3(X,\bigwedge^2\tangent_X)=0$, hence the claim.

  \paragraph{Second step: blowup computation}
  We consider the short exact sequence cutting out~$Z$ inside~$Y$ and tensor it with~$\Omega_Y^1\otimes\omega_Y^\vee$ to obtain
  \begin{equation}
    0\to\mathcal{I}_Z\otimes\Omega_Y^1\otimes\omega_Y^\vee\to\Omega_Y^1\otimes\omega_Y^\vee\to(\Omega_Y^1\otimes\omega_Y^\vee)|_Z\to 0.
  \end{equation}
  The cohomology of the middle term is determined by the methods in \cref{subsection:toric} and is given by
  \begin{equation}
    \label{equation:2-1-to-1-11}
    \hh^i(Y,\Omega_Y^1\otimes\omega_Y^\vee)=3,0,7\quad\text{for $i=0,1,2$}.
  \end{equation}
  To compute the cohomology of the third term, consider the conormal sequence for~$Z$ in~$Y$ twisted by~$\omega_Y^\vee|_Z\cong\mathcal{O}_Z(2)$ (because~$Y$ is of index~2), which, since $Z$ is an elliptic curve, reads
  \begin{equation}
    0\to\mathcal{I}_Z/\mathcal{I}_Z^2\otimes\mathcal{O}_Z(2)\to(\Omega_Y^1\otimes\omega_Y^\vee)|_Z\to\mathcal{O}_Z(2)\to 0.
  \end{equation}
  As~$Z$ is cut out by two half-anticanonical divisors, we obtain that~$\mathcal{I}_Z/\mathcal{I}_Z^2\otimes\mathcal{O}_Z(2)\cong\mathcal{O}_Z(1)^{\oplus2}$. It now suffices to compute the cohomology of the line bundles~$\mathcal{O}_Z(1)$ and~$\mathcal{O}_Z(2)$, which is concentrated in degree~0 by degree reasons, and is determined by a Hilbert series computation on the weighted projective space. We obtain~$\hh^i(Z,(\Omega_Y^1\otimes\omega_Y^\vee)|_Z)=4,0$ for~$i=0,1$. Combining this with \eqref{equation:2-1-to-1-11} we get that
  \begin{equation}
    \label{equation:2-1-blowup-claim}
    \left\{
    \begin{aligned}
      \hh^0(X,\Omega_Y^1\otimes\omega_Y^\vee\otimes\mathcal{I}_Z)-\hh^1(X,\Omega_Y^1\otimes\omega_Y^\vee\otimes\mathcal{I}_Z)&=-1 \\
      \hh^2(X,\Omega_Y^1\otimes\omega_Y^\vee\otimes\mathcal{I}_Z)&=7.
    \end{aligned}
    \right.
  \end{equation}
  Together with \eqref{equation:2-1-weighted-claim} and \cref{corollary:blowup-identification} this finishes the computation.
\end{proof}

Next we tackle the underdetermined case \MM{2}{3} in exactly the same way. In the proof we will explain which details of the computation change.
\begin{proposition}
  \label{proposition:2-3}
  Let~$X$ be a Fano 3-fold in the deformation family~\MM{2}{3}. Then we have that~$\hh^i(X,\bigwedge^2\tangent_X)=1,3,1$ for~$i=0,1,2$.
\end{proposition}

\begin{proof}
  Such a variety is described as the blowup of a Fano 3-fold~$Y$ of deformation type~\MM{1}{12} in an elliptic curve obtained as complete intersection of two half-anticanonical divisors, and~$Y$ is given as a quartic hypersurface in the weighted projective space~$\mathbb{P}\colonequals\mathbb{P}(1^4,2)$.

  Alternatively,~$X$ is a~$(1,1)$-section of~$F=Y\times \mathbb{P}^1$. Computing~$\hh^i(X,\bigwedge^2\tangent_X) = \hh^i(X,\Omega^1_X(1,1))$ is analogous to the two-step procedure from \cref{proposition:2-1}, so we only summarise the relevant differences in the numerology appearing.

  \paragraph{First step: weighted projective space computation (abbreviated)}
  This goes along the same lines as the first step in the proof of \cref{proposition:2-1}, except that the claim \eqref{equation:2-1-weighted-claim} now becomes
  \begin{equation}
    \label{equation:2-3-weighted-claim}
    \left\{
    \begin{aligned}
      \hh^0(X,\Omega^1_X(1,1))&=1 \\
      \hh^2(X,\Omega^1_X(1,1))-\hh^1(X,\Omega^1_X(1,1))&=-2.
    \end{aligned}
   \right.
  \end{equation}
  From the Hodge numbers of~$\mathbb{P}^1$ and~\MM{1}{12} we now get~$\hh^i(F,\Omega^1_F)=0,2,10,0,0$ for~$i=0,1,2,3,4$. The twisted conormal sequence for~$Y$, which is a quartic hypersurface in~$\mathbb{P}$, gives~$\hh^3(Y,\mathcal{O}_Y(-3)) = 4$. We obtain
  \begin{equation}
    \begin{alignedat}{2}
      \hh^0(Y,\Omega^1_Y(1))&= 0, &\qquad \hh^1(Y,\Omega^1_Y(1))&= 0, \\
      \hh^2(Y,\Omega^1_Y(1))&= \alpha + 4, & \qquad\hh^3(Y,\Omega^1_Y(1))&= \alpha
    \end{alignedat}
  \end{equation}
  for some integer $\alpha\geq 0$. As in the proof of \cref{proposition:2-1} we get
  \begin{equation}
    \begin{alignedat}{2}
      \hh^0(X,\Omega^1_F(1,1)|_X)&=2, & \qquad\hh^1(X,\Omega^1_F(1,1)|_X) &= \beta, \\
      \hh^2(X,\Omega^1_F(1,1)|_X)&=\beta + 2\alpha + (2\cdot4-10), & \qquad\hh^3(X,\Omega^1_F(1,1)|_X)&=2\alpha
    \end{alignedat}
  \end{equation}
  for some integers $\alpha, \beta\geq0$. But by \cref{lemma:kodaira} we get $2\alpha=\hh^3(X,\bigwedge^2\tangent_X)=0$, hence the claim.

  \paragraph{Second step: blowup computation}
  This goes along the same lines as the second step in the proof of \cref{proposition:2-1}, except that \eqref{equation:2-1-to-1-11} now becomes
  \begin{equation}
    \label{equation:2-3-to-1-12}
    \hh^i(Y,\Omega_Y^1\otimes\omega_Y^\vee)=6,0,1\quad\text{for $i=0,1,2$}
  \end{equation}
  and that the Hilbert series computation is performed for an elliptic curve in the weighted projective space~$\mathbb{P}(1^4,2)$. We obtain~$\hh^i(Z,(\Omega_Y^1\otimes\omega_Y^\vee)|_Z)=8,0$ for~$i=0,1$. Combining this with \eqref{equation:2-3-to-1-12} we get that
  \begin{equation}
    \left\{
    \label{equation:2-3-blowup-claim}
    \begin{aligned}
      \hh^0(Y,\Omega_Y^1\otimes\omega_Y^\vee\otimes\mathcal{I}_Z)-\hh^1(Y,\Omega_Y^1\otimes\omega_Y^\vee\otimes\mathcal{I}_Z)&=-2 \\
      \hh^2(Y,\Omega_Y^1\otimes\omega_Y^\vee\otimes\mathcal{I}_Z)&=1.
    \end{aligned}
    \right.
  \end{equation}
  Together with \eqref{equation:2-3-weighted-claim} and \cref{corollary:blowup-identification} this finishes the computation.
\end{proof}

The next underdetermined case is the Fano 3-fold which was originally missed in the classification \cite{MR1969009}. The method is similar to what we did for \MM{2}{1} and \MM{2}{3} using the birational description, but requires less work.
\begin{proposition}
  \label{proposition:4-13}
  Let~$X$ be a Fano 3-fold in the deformation family~\MM{4}{13}. Then we have that~$\hh^i(X,\bigwedge^2\tangent_X)=4,0,0$ for~$i=0,1,2$.
\end{proposition}

\begin{proof}
  Such a variety is described as the blowup of the Fano 3-fold~$Y=\mathbb{P}^1\times\mathbb{P}^1\times\mathbb{P}^1$ of deformation type \MM{3}{27} in a rational curve~$Z$ of tridegree~$(1,1,3)$. The curve~$Z$ is given as a complete intersection of type~$(2,1,1)$ and~$(1,1,0)$, see \cite[\S86]{MR3470714}.

  By \cref{corollary:blowup-identification} we want to compute~$\HH^\bullet(Y,\Omega_Y^1\otimes\omega_Y^\vee\otimes\mathcal{I}_Z)$. Denote by~$p_{i,j}\colon Y\to\mathbb{P}^1\times\mathbb{P}^1$ the three projections. Using the isomorphism
  \begin{equation}
    \Omega_Y^1\otimes\omega_Y^\vee\cong\bigoplus_{1\leq i<j\leq 3}p_{i,j}^*(\mathcal{O}_{\mathbb{P}^1\times\mathbb{P}^1}(2,2))
  \end{equation}
  and the projection formula we want to compute the cohomology of
  \begin{equation}
    \mathcal{O}_{\mathbb{P}^1\times\mathbb{P}^1}(2,2)\otimes p_{i,j,*}\mathcal{I}_Z.
  \end{equation}
  But~$p_{i,j,*}\mathcal{I}_Z=\mathcal{I}_{Z_{i,j}}$ where~$Z_{i,j}$ is the image of~$Z$ under~$p_{i,j}$, which is in turn a divisor of bidegree~$(1,1)$, resp.~$(1,3)$ and~$(1,3)$ on~$\mathbb{P}^1\times\mathbb{P}^1$. This reduces the computation to the cohomology of~$\mathcal{O}_{\mathbb{P}^1\times\mathbb{P}^1}(1,1)$ and~$\mathcal{O}_{\mathbb{P}^1\times\mathbb{P}^1}(1,-1)^{\oplus2}$, where the latter is cohomology-free and the former has cohomology concentrated in degree~0, where it is 4-dimensional.
\end{proof}

The final case is the product of a del Pezzo surface of degree~1 with~$\mathbb{P}^1$, and the computation is immediate (e.g.~using the description of \cref{section:surfaces} and the K\"unneth formula).

\begin{proposition}
  \label{proposition:10-1}
  Let~$X$ be a Fano 3-fold in the deformation family~\MM{10}{1}. Then we have that~$\hh^i(X,\bigwedge^2\tangent_X)=2,24,0$ for~$i=0,1,2$.
\end{proposition}

\clearpage

\newgeometry{left=1.5cm, right=1.5cm, top=1.5cm, bottom=1.5cm}

\appendix
\section{Dimensions for Fano 3-folds}
\label{section:values}
In this appendix we have collected all the polyvector parallelograms of Fano 3-folds. The notation we use for this is introduced in \eqref{equation:parallelogram}: writing~$\pv^{p,q}\colonequals\dim_k\HH^p(X,\bigwedge^q\tangent_X)$, and using that~$X$ is a Fano 3-fold, we can summarise the Hochschild--Kostant--Rosenberg decomposition of Hochschild cohomology as
\begin{equation}
  \arraycolsep=10pt
  \renewcommand\arraystretch{1.2}
  \begin{array}{*5{>{}c}} {\scriptstyle\HHHH^0(X)} & 1 \\ {\scriptstyle\HHHH^1(X)} & 0 & \pv^{0,1} \\ {\scriptstyle\HHHH^2(X)} & 0 & \pv^{1,1} & \pv^{0,2} \\ {\scriptstyle\HHHH^3(X)} & 0 & 0 & \pv^{1,2} & \pv^{0,3} \\ {\scriptstyle\HHHH^4(X)} & & 0 & \pv^{2,2} & 0 \\ {\scriptstyle\HHHH^5(X)} & && 0 & 0 \\ {\scriptstyle\HHHH^6(X)} & &&& 0 \end{array}
\end{equation}
In the case of jumping of~$\hh^i(X,\tangent_X)$, we have given the lowest value, and indicated with a * next to the value~$\hh^0(X,\tangent_X)$ that jumping occurs.

This information (and more) is also available in a more interactive way at \cite{fanography}.

\arraycolsep=3pt
\paragraph{Rank 1}
\begin{flushleft}
\begin{minipage}{.166\linewidth}
  \footnotesize
  \centering
  $\parallelogramthreefold{0}{68}{0}{0}{35}{4}$
  {Polyvectors \\ of \MM{1}{1}}
\end{minipage}\begin{minipage}{.166\linewidth}
  \footnotesize
  \centering
  $\parallelogramthreefold{0}{45}{0}{0}{15}{5}$
  {Polyvectors \\ of \MM{1}{2}}
\end{minipage}\begin{minipage}{.166\linewidth}
  \footnotesize
  \centering
  $\parallelogramthreefold{0}{34}{0}{0}{7}{6}$
  {Polyvectors \\ of \MM{1}{3}}
\end{minipage}\begin{minipage}{.166\linewidth}
  \footnotesize
  \centering
  $\parallelogramthreefold{0}{27}{0}{0}{3}{7}$
  {Polyvectors \\ of \MM{1}{4}}
\end{minipage}\begin{minipage}{.166\linewidth}
  \footnotesize
  \centering
  $\parallelogramthreefold{0}{22}{0}{0}{1}{8}$
  {Polyvectors \\ of \MM{1}{5}}
\end{minipage}\begin{minipage}{.166\linewidth}
  \footnotesize
  \centering
  $\parallelogramthreefold{0}{18}{0}{0}{0}{9}$
  {Polyvectors \\ of \MM{1}{6}}
\end{minipage}\vspace{0.5cm}
\begin{minipage}{.166\linewidth}
  \footnotesize
  \centering
  $\parallelogramthreefold{0}{15}{0}{0}{0}{10}$
  {Polyvectors \\ of \MM{1}{7}}
\end{minipage}\begin{minipage}{.166\linewidth}
  \footnotesize
  \centering
  $\parallelogramthreefold{0}{12}{0}{0}{0}{11}$
  {Polyvectors \\ of \MM{1}{8}}
\end{minipage}\begin{minipage}{.166\linewidth}
  \footnotesize
  \centering
  $\parallelogramthreefold{0}{10}{1}{0}{0}{12}$
  {Polyvectors \\ of \MM{1}{9}}
\end{minipage}\begin{minipage}{.166\linewidth}
  \footnotesize
  \centering
  $\parallelogramthreefold{0^*}{6}{3}{0}{0}{14}$
  {Polyvectors \\ of \MM{1}{10}}
\end{minipage}\begin{minipage}{.166\linewidth}
  \footnotesize
  \centering
  $\parallelogramthreefold{0}{34}{3}{0}{7}{7}$
  {Polyvectors \\ of \MM{1}{11}}
\end{minipage}\begin{minipage}{.166\linewidth}
  \footnotesize
  \centering
  $\parallelogramthreefold{0}{19}{6}{0}{1}{11}$
  {Polyvectors \\ of \MM{1}{12}}
\end{minipage}\vspace{0.5cm}
\begin{minipage}{.166\linewidth}
  \footnotesize
  \centering
  $\parallelogramthreefold{0}{10}{10}{0}{0}{15}$
  {Polyvectors \\ of \MM{1}{13}}
\end{minipage}\begin{minipage}{.166\linewidth}
  \footnotesize
  \centering
  $\parallelogramthreefold{0}{3}{15}{0}{0}{19}$
  {Polyvectors \\ of \MM{1}{14}}
\end{minipage}\begin{minipage}{.166\linewidth}
  \footnotesize
  \centering
  $\parallelogramthreefold{3}{0}{21}{0}{0}{23}$
  {Polyvectors \\ of \MM{1}{15}}
\end{minipage}\begin{minipage}{.166\linewidth}
  \footnotesize
  \centering
  $\parallelogramthreefold{10}{0}{35}{0}{0}{30}$
  {Polyvectors \\ of \MM{1}{16}}
\end{minipage}\begin{minipage}{.166\linewidth}
  \footnotesize
  \centering
  $\parallelogramthreefold{15}{0}{45}{0}{0}{35}$
  {Polyvectors \\ of \MM{1}{17}}
\end{minipage}\end{flushleft}

\clearpage
\paragraph{Rank 2}
\begin{flushleft}
\begin{minipage}{.166\linewidth}
  \footnotesize
  \centering
  $\parallelogramthreefold{0}{36}{1}{2}{7}{5}$
  {Polyvectors \\ of \MM{2}{1}}
\end{minipage}\begin{minipage}{.166\linewidth}
  \footnotesize
  \centering
  $\parallelogramthreefold{0}{33}{0}{0}{6}{6}$
  {Polyvectors \\ of \MM{2}{2}}
\end{minipage}\begin{minipage}{.166\linewidth}
  \footnotesize
  \centering
  $\parallelogramthreefold{0}{23}{1}{3}{1}{7}$
  {Polyvectors \\ of \MM{2}{3}}
\end{minipage}\begin{minipage}{.166\linewidth}
  \footnotesize
  \centering
  $\parallelogramthreefold{0}{21}{0}{0}{0}{8}$
  {Polyvectors \\ of \MM{2}{4}}
\end{minipage}\begin{minipage}{.166\linewidth}
  \footnotesize
  \centering
  $\parallelogramthreefold{0}{16}{1}{3}{0}{9}$
  {Polyvectors \\ of \MM{2}{5}}
\end{minipage}\begin{minipage}{.166\linewidth}
  \footnotesize
  \centering
  $\parallelogramthreefold{0}{19}{0}{0}{1}{9}$
  {Polyvectors \\ of \MM{2}{6}}
\end{minipage}\vspace{0.5cm}
\begin{minipage}{.166\linewidth}
  \footnotesize
  \centering
  $\parallelogramthreefold{0}{14}{0}{1}{0}{10}$
  {Polyvectors \\ of \MM{2}{7}}
\end{minipage}\begin{minipage}{.166\linewidth}
  \footnotesize
  \centering
  $\parallelogramthreefold{0}{18}{3}{1}{1}{10}$
  {Polyvectors \\ of \MM{2}{8}}
\end{minipage}\begin{minipage}{.166\linewidth}
  \footnotesize
  \centering
  $\parallelogramthreefold{0}{13}{1}{0}{0}{11}$
  {Polyvectors \\ of \MM{2}{9}}
\end{minipage}\begin{minipage}{.166\linewidth}
  \footnotesize
  \centering
  $\parallelogramthreefold{0}{11}{1}{2}{0}{11}$
  {Polyvectors \\ of \MM{2}{10}}
\end{minipage}\begin{minipage}{.166\linewidth}
  \footnotesize
  \centering
  $\parallelogramthreefold{0}{12}{3}{0}{0}{12}$
  {Polyvectors \\ of \MM{2}{11}}
\end{minipage}\begin{minipage}{.166\linewidth}
  \footnotesize
  \centering
  $\parallelogramthreefold{0}{9}{3}{0}{0}{13}$
  {Polyvectors \\ of \MM{2}{12}}
\end{minipage}\vspace{0.5cm}
\begin{minipage}{.166\linewidth}
  \footnotesize
  \centering
  $\parallelogramthreefold{0}{8}{2}{0}{0}{13}$
  {Polyvectors \\ of \MM{2}{13}}
\end{minipage}\begin{minipage}{.166\linewidth}
  \footnotesize
  \centering
  $\parallelogramthreefold{0}{7}{1}{0}{0}{13}$
  {Polyvectors \\ of \MM{2}{14}}
\end{minipage}\begin{minipage}{.166\linewidth}
  \footnotesize
  \centering
  $\parallelogramthreefold{0}{9}{6}{0}{0}{14}$
  {Polyvectors \\ of \MM{2}{15}}
\end{minipage}\begin{minipage}{.166\linewidth}
  \footnotesize
  \centering
  $\parallelogramthreefold{0}{7}{4}{0}{0}{14}$
  {Polyvectors \\ of \MM{2}{16}}
\end{minipage}\begin{minipage}{.166\linewidth}
  \footnotesize
  \centering
  $\parallelogramthreefold{0}{5}{5}{0}{0}{15}$
  {Polyvectors \\ of \MM{2}{17}}
\end{minipage}\begin{minipage}{.166\linewidth}
  \footnotesize
  \centering
  $\parallelogramthreefold{0}{6}{6}{0}{0}{15}$
  {Polyvectors \\ of \MM{2}{18}}
\end{minipage}\vspace{0.5cm}
\begin{minipage}{.166\linewidth}
  \footnotesize
  \centering
  $\parallelogramthreefold{0}{5}{8}{0}{0}{16}$
  {Polyvectors \\ of \MM{2}{19}}
\end{minipage}\begin{minipage}{.166\linewidth}
  \footnotesize
  \centering
  $\parallelogramthreefold{0^*}{3}{6}{0}{0}{16}$
  {Polyvectors \\ of \MM{2}{20}}
\end{minipage}\begin{minipage}{.166\linewidth}
  \footnotesize
  \centering
  $\parallelogramthreefold{0^*}{2}{8}{0}{0}{17}$
  {Polyvectors \\ of \MM{2}{21}}
\end{minipage}\begin{minipage}{.166\linewidth}
  \footnotesize
  \centering
  $\parallelogramthreefold{0^*}{1}{10}{0}{0}{18}$
  {Polyvectors \\ of \MM{2}{22}}
\end{minipage}\begin{minipage}{.166\linewidth}
  \footnotesize
  \centering
  $\parallelogramthreefold{0}{2}{11}{0}{0}{18}$
  {Polyvectors \\ of \MM{2}{23}}
\end{minipage}\begin{minipage}{.166\linewidth}
  \footnotesize
  \centering
  $\parallelogramthreefold{0^*}{1}{10}{0}{0}{18}$
  {Polyvectors \\ of \MM{2}{24}}
\end{minipage}\vspace{0.5cm}
\begin{minipage}{.166\linewidth}
  \footnotesize
  \centering
  $\parallelogramthreefold{0}{1}{13}{0}{0}{19}$
  {Polyvectors \\ of \MM{2}{25}}
\end{minipage}\begin{minipage}{.166\linewidth}
  \footnotesize
  \centering
  $\parallelogramthreefold{1^*}{0}{14}{0}{0}{20}$
  {Polyvectors \\ of \MM{2}{26}}
\end{minipage}\begin{minipage}{.166\linewidth}
  \footnotesize
  \centering
  $\parallelogramthreefold{3}{0}{18}{0}{0}{22}$
  {Polyvectors \\ of \MM{2}{27}}
\end{minipage}\begin{minipage}{.166\linewidth}
  \footnotesize
  \centering
  $\parallelogramthreefold{4}{1}{21}{0}{0}{23}$
  {Polyvectors \\ of \MM{2}{28}}
\end{minipage}\begin{minipage}{.166\linewidth}
  \footnotesize
  \centering
  $\parallelogramthreefold{4}{0}{20}{0}{0}{23}$
  {Polyvectors \\ of \MM{2}{29}}
\end{minipage}\begin{minipage}{.166\linewidth}
  \footnotesize
  \centering
  $\parallelogramthreefold{7}{0}{26}{0}{0}{26}$
  {Polyvectors \\ of \MM{2}{30}}
\end{minipage}\vspace{0.5cm}
\begin{minipage}{.166\linewidth}
  \footnotesize
  \centering
  $\parallelogramthreefold{7}{0}{26}{0}{0}{26}$
  {Polyvectors \\ of \MM{2}{31}}
\end{minipage}\begin{minipage}{.166\linewidth}
  \footnotesize
  \centering
  $\parallelogramthreefold{8}{0}{28}{0}{0}{27}$
  {Polyvectors \\ of \MM{2}{32}}
\end{minipage}\begin{minipage}{.166\linewidth}
  \footnotesize
  \centering
  $\parallelogramthreefold{11}{0}{34}{0}{0}{30}$
  {Polyvectors \\ of \MM{2}{33}}
\end{minipage}\begin{minipage}{.166\linewidth}
  \footnotesize
  \centering
  $\parallelogramthreefold{11}{0}{34}{0}{0}{30}$
  {Polyvectors \\ of \MM{2}{34}}
\end{minipage}\begin{minipage}{.166\linewidth}
  \footnotesize
  \centering
  $\parallelogramthreefold{12}{0}{36}{0}{0}{31}$
  {Polyvectors \\ of \MM{2}{35}}
\end{minipage}\begin{minipage}{.166\linewidth}
  \footnotesize
  \centering
  $\parallelogramthreefold{15}{0}{42}{0}{0}{34}$
  {Polyvectors \\ of \MM{2}{36}}
\end{minipage}\vspace{0.5cm}
\end{flushleft}

\clearpage
\paragraph{Rank 3}
\begin{flushleft}
\begin{minipage}{.166\linewidth}
  \footnotesize
  \centering
  $\parallelogramthreefold{0}{17}{0}{2}{1}{9}$
  {Polyvectors \\ of \MM{3}{1}}
\end{minipage}\begin{minipage}{.166\linewidth}
  \footnotesize
  \centering
  $\parallelogramthreefold{0}{11}{2}{6}{0}{10}$
  {Polyvectors \\ of \MM{3}{2}}
\end{minipage}\begin{minipage}{.166\linewidth}
  \footnotesize
  \centering
  $\parallelogramthreefold{0}{9}{0}{0}{0}{12}$
  {Polyvectors \\ of \MM{3}{3}}
\end{minipage}\begin{minipage}{.166\linewidth}
  \footnotesize
  \centering
  $\parallelogramthreefold{0}{8}{2}{3}{0}{12}$
  {Polyvectors \\ of \MM{3}{4}}
\end{minipage}\begin{minipage}{.166\linewidth}
  \footnotesize
  \centering
  $\parallelogramthreefold{0^*}{5}{3}{4}{0}{13}$
  {Polyvectors \\ of \MM{3}{5}}
\end{minipage}\begin{minipage}{.166\linewidth}
  \footnotesize
  \centering
  $\parallelogramthreefold{0}{5}{2}{0}{0}{14}$
  {Polyvectors \\ of \MM{3}{6}}
\end{minipage}\vspace{0.5cm}
\begin{minipage}{.166\linewidth}
  \footnotesize
  \centering
  $\parallelogramthreefold{0}{4}{4}{0}{0}{15}$
  {Polyvectors \\ of \MM{3}{7}}
\end{minipage}\begin{minipage}{.166\linewidth}
  \footnotesize
  \centering
  $\parallelogramthreefold{0^*}{3}{3}{0}{0}{15}$
  {Polyvectors \\ of \MM{3}{8}}
\end{minipage}\begin{minipage}{.166\linewidth}
  \footnotesize
  \centering
  $\parallelogramthreefold{1}{6}{8}{0}{0}{16}$
  {Polyvectors \\ of \MM{3}{9}}
\end{minipage}\begin{minipage}{.166\linewidth}
  \footnotesize
  \centering
  $\parallelogramthreefold{0^*}{2}{5}{0}{0}{16}$
  {Polyvectors \\ of \MM{3}{10}}
\end{minipage}\begin{minipage}{.166\linewidth}
  \footnotesize
  \centering
  $\parallelogramthreefold{0}{2}{8}{0}{0}{17}$
  {Polyvectors \\ of \MM{3}{11}}
\end{minipage}\begin{minipage}{.166\linewidth}
  \footnotesize
  \centering
  $\parallelogramthreefold{0^*}{1}{7}{0}{0}{17}$
  {Polyvectors \\ of \MM{3}{12}}
\end{minipage}\vspace{0.5cm}
\begin{minipage}{.166\linewidth}
  \footnotesize
  \centering
  $\parallelogramthreefold{1^*}{1}{9}{0}{0}{18}$
  {Polyvectors \\ of \MM{3}{13}}
\end{minipage}\begin{minipage}{.166\linewidth}
  \footnotesize
  \centering
  $\parallelogramthreefold{1}{1}{12}{0}{0}{19}$
  {Polyvectors \\ of \MM{3}{14}}
\end{minipage}\begin{minipage}{.166\linewidth}
  \footnotesize
  \centering
  $\parallelogramthreefold{1}{0}{11}{0}{0}{19}$
  {Polyvectors \\ of \MM{3}{15}}
\end{minipage}\begin{minipage}{.166\linewidth}
  \footnotesize
  \centering
  $\parallelogramthreefold{2}{0}{13}{0}{0}{20}$
  {Polyvectors \\ of \MM{3}{16}}
\end{minipage}\begin{minipage}{.166\linewidth}
  \footnotesize
  \centering
  $\parallelogramthreefold{3}{0}{15}{0}{0}{21}$
  {Polyvectors \\ of \MM{3}{17}}
\end{minipage}\begin{minipage}{.166\linewidth}
  \footnotesize
  \centering
  $\parallelogramthreefold{3}{0}{15}{0}{0}{21}$
  {Polyvectors \\ of \MM{3}{18}}
\end{minipage}\vspace{0.5cm}
\begin{minipage}{.166\linewidth}
  \footnotesize
  \centering
  $\parallelogramthreefold{4}{0}{17}{0}{0}{22}$
  {Polyvectors \\ of \MM{3}{19}}
\end{minipage}\begin{minipage}{.166\linewidth}
  \footnotesize
  \centering
  $\parallelogramthreefold{4}{0}{17}{0}{0}{22}$
  {Polyvectors \\ of \MM{3}{20}}
\end{minipage}\begin{minipage}{.166\linewidth}
  \footnotesize
  \centering
  $\parallelogramthreefold{4}{0}{17}{0}{0}{22}$
  {Polyvectors \\ of \MM{3}{21}}
\end{minipage}\begin{minipage}{.166\linewidth}
  \footnotesize
  \centering
  $\parallelogramthreefold{5}{0}{19}{0}{0}{23}$
  {Polyvectors \\ of \MM{3}{22}}
\end{minipage}\begin{minipage}{.166\linewidth}
  \footnotesize
  \centering
  $\parallelogramthreefold{6}{0}{21}{0}{0}{24}$
  {Polyvectors \\ of \MM{3}{23}}
\end{minipage}\begin{minipage}{.166\linewidth}
  \footnotesize
  \centering
  $\parallelogramthreefold{6}{0}{21}{0}{0}{24}$
  {Polyvectors \\ of \MM{3}{24}}
\end{minipage}\vspace{0.5cm}
\begin{minipage}{.166\linewidth}
  \footnotesize
  \centering
  $\parallelogramthreefold{7}{0}{23}{0}{0}{25}$
  {Polyvectors \\ of \MM{3}{25}}
\end{minipage}\begin{minipage}{.166\linewidth}
  \footnotesize
  \centering
  $\parallelogramthreefold{8}{0}{25}{0}{0}{26}$
  {Polyvectors \\ of \MM{3}{26}}
\end{minipage}\begin{minipage}{.166\linewidth}
  \footnotesize
  \centering
  $\parallelogramthreefold{9}{0}{27}{0}{0}{27}$
  {Polyvectors \\ of \MM{3}{27}}
\end{minipage}\begin{minipage}{.166\linewidth}
  \footnotesize
  \centering
  $\parallelogramthreefold{9}{0}{27}{0}{0}{27}$
  {Polyvectors \\ of \MM{3}{28}}
\end{minipage}\begin{minipage}{.166\linewidth}
  \footnotesize
  \centering
  $\parallelogramthreefold{10}{0}{29}{0}{0}{28}$
  {Polyvectors \\ of \MM{3}{29}}
\end{minipage}\begin{minipage}{.166\linewidth}
  \footnotesize
  \centering
  $\parallelogramthreefold{10}{0}{29}{0}{0}{28}$
  {Polyvectors \\ of \MM{3}{30}}
\end{minipage}\vspace{0.5cm}
\begin{minipage}{.166\linewidth}
  \footnotesize
  \centering
  $\parallelogramthreefold{11}{0}{31}{0}{0}{29}$
  {Polyvectors \\ of \MM{3}{31}}
\end{minipage}\end{flushleft}

\clearpage
\paragraph{Rank 4}
\begin{flushleft}
\begin{minipage}{.166\linewidth}
  \footnotesize
  \centering
  $\parallelogramthreefold{0}{3}{3}{0}{0}{15}$
  {Polyvectors \\ of \MM{4}{1}}
\end{minipage}\begin{minipage}{.166\linewidth}
  \footnotesize
  \centering
  $\parallelogramthreefold{1}{2}{7}{0}{0}{17}$
  {Polyvectors \\ of \MM{4}{2}}
\end{minipage}\begin{minipage}{.166\linewidth}
  \footnotesize
  \centering
  $\parallelogramthreefold{1}{0}{8}{0}{0}{18}$
  {Polyvectors \\ of \MM{4}{3}}
\end{minipage}\begin{minipage}{.166\linewidth}
  \footnotesize
  \centering
  $\parallelogramthreefold{2}{0}{10}{0}{0}{19}$
  {Polyvectors \\ of \MM{4}{4}}
\end{minipage}\begin{minipage}{.166\linewidth}
  \footnotesize
  \centering
  $\parallelogramthreefold{2}{0}{10}{0}{0}{19}$
  {Polyvectors \\ of \MM{4}{5}}
\end{minipage}\begin{minipage}{.166\linewidth}
  \footnotesize
  \centering
  $\parallelogramthreefold{3}{0}{12}{0}{0}{20}$
  {Polyvectors \\ of \MM{4}{6}}
\end{minipage}\vspace{0.5cm}
\begin{minipage}{.166\linewidth}
  \footnotesize
  \centering
  $\parallelogramthreefold{4}{0}{14}{0}{0}{21}$
  {Polyvectors \\ of \MM{4}{7}}
\end{minipage}\begin{minipage}{.166\linewidth}
  \footnotesize
  \centering
  $\parallelogramthreefold{5}{0}{16}{0}{0}{22}$
  {Polyvectors \\ of \MM{4}{8}}
\end{minipage}\begin{minipage}{.166\linewidth}
  \footnotesize
  \centering
  $\parallelogramthreefold{6}{0}{18}{0}{0}{23}$
  {Polyvectors \\ of \MM{4}{9}}
\end{minipage}\begin{minipage}{.166\linewidth}
  \footnotesize
  \centering
  $\parallelogramthreefold{7}{0}{20}{0}{0}{24}$
  {Polyvectors \\ of \MM{4}{10}}
\end{minipage}\begin{minipage}{.166\linewidth}
  \footnotesize
  \centering
  $\parallelogramthreefold{8}{0}{22}{0}{0}{25}$
  {Polyvectors \\ of \MM{4}{11}}
\end{minipage}\begin{minipage}{.166\linewidth}
  \footnotesize
  \centering
  $\parallelogramthreefold{9}{0}{24}{0}{0}{26}$
  {Polyvectors \\ of \MM{4}{12}}
\end{minipage}\vspace{0.5cm}
\begin{minipage}{.166\linewidth}
  \footnotesize
  \centering
  $\parallelogramthreefold{0^*}{1}{4}{0}{0}{16}$
  {Polyvectors \\ of \MM{4}{13}}
\end{minipage}\end{flushleft}

\clearpage
\paragraph{Rank 5}
\begin{flushleft}
\begin{minipage}{.166\linewidth}
  \footnotesize
  \centering
  $\parallelogramthreefold{1}{0}{5}{0}{0}{17}$
  {Polyvectors \\ of \MM{5}{1}}
\end{minipage}\begin{minipage}{.166\linewidth}
  \footnotesize
  \centering
  $\parallelogramthreefold{5}{0}{13}{0}{0}{21}$
  {Polyvectors \\ of \MM{5}{2}}
\end{minipage}\begin{minipage}{.166\linewidth}
  \footnotesize
  \centering
  $\parallelogramthreefold{5}{0}{13}{0}{0}{21}$
  {Polyvectors \\ of \MM{5}{3}}
\end{minipage}\end{flushleft}

\paragraph{Rank 6}
\begin{flushleft}
\begin{minipage}{.166\linewidth}
  \footnotesize
  \centering
  $\parallelogramthreefold{3}{0}{6}{0}{0}{18}$
  {Polyvectors \\ of \MM{6}{1}}
\end{minipage}\end{flushleft}

\paragraph{Rank 7}
\begin{flushleft}
\begin{minipage}{.166\linewidth}
  \footnotesize
  \centering
  $\parallelogramthreefold{3}{2}{5}{6}{0}{15}$
  {Polyvectors \\ of \MM{7}{1}}
\end{minipage}\end{flushleft}

\paragraph{Rank 8}
\begin{flushleft}
\begin{minipage}{.166\linewidth}
  \footnotesize
  \centering
  $\parallelogramthreefold{3}{4}{4}{12}{0}{12}$
  {Polyvectors \\ of \MM{8}{1}}
\end{minipage}\end{flushleft}

\paragraph{Rank 9}
\begin{flushleft}
\begin{minipage}{.166\linewidth}
  \footnotesize
  \centering
  $\parallelogramthreefold{3}{6}{3}{18}{0}{9}$
  {Polyvectors \\ of \MM{9}{1}}
\end{minipage}\end{flushleft}

\paragraph{Rank 10}
\begin{flushleft}
\begin{minipage}{.166\linewidth}
  \footnotesize
  \centering
  $\parallelogramthreefold{3}{8}{2}{24}{0}{6}$
  {Polyvectors \\ of \MM{10}{1}}
\end{minipage}\end{flushleft}

\restoregeometry

\section{Dimensions for del Pezzo surfaces}
\label{section:surfaces}
For ease of reference and completeness' sake we give the Hochschild--Kostant--Rosenberg decomposition of the Hochschild cohomology of del Pezzo surfaces. We use the notation introduced in \eqref{equation:parallelogram}, suitably modified for surfaces.

One can compute this in many ways, e.g.~using the methods from \cref{subsection:general-results}. There is no cohomology jumping in this case. What is interesting to remark is that one does observe cohomology jumping when extending to noncommutative del Pezzo surfaces, as in \cite{MR3988086}.

\begin{flushleft}
\begin{minipage}{.2\linewidth}
  \footnotesize
  \centering
  $\parallelogramsurface{8}{0}{10}$

  {Polyvectors \\ of $\mathbb{P}^2$}
\end{minipage}\begin{minipage}{.2\linewidth}
  \footnotesize
  \centering
  $\parallelogramsurface{6}{0}{9}$

  {Polyvectors \\ of $\mathbb{P}^1\times\mathbb{P}^1$}
\end{minipage}\begin{minipage}{.2\linewidth}
  \footnotesize
  \centering
  $\parallelogramsurface{6}{0}{9}$

  {Polyvectors \\ of $\Bl_1\mathbb{P}^2$}
\end{minipage}\begin{minipage}{.2\linewidth}
  \footnotesize
  \centering
  $\parallelogramsurface{4}{0}{8}$

  {Polyvectors \\ of $\Bl_2\mathbb{P}^2$}
\end{minipage}\begin{minipage}{.2\linewidth}
  \footnotesize
  \centering
  $\parallelogramsurface{2}{0}{7}$

  {Polyvectors \\ of $\Bl_3\mathbb{P}^2$}
\end{minipage}\vspace{0.5cm}
\begin{minipage}{.2\linewidth}
  \footnotesize
  \centering
  $\parallelogramsurface{0}{0}{6}$

  {Polyvectors \\ of $\Bl_4\mathbb{P}^2$}
\end{minipage}\begin{minipage}{.2\linewidth}
  \footnotesize
  \centering
  $\parallelogramsurface{0}{2}{5}$

  {Polyvectors \\ of $\Bl_5\mathbb{P}^2$}
\end{minipage}\begin{minipage}{.2\linewidth}
  \footnotesize
  \centering
  $\parallelogramsurface{0}{4}{4}$

  {Polyvectors \\ of $\Bl_6\mathbb{P}^2$}
\end{minipage}\begin{minipage}{.2\linewidth}
  \footnotesize
  \centering
  $\parallelogramsurface{0}{6}{3}$

  {Polyvectors \\ of $\Bl_7\mathbb{P}^2$}
\end{minipage}\begin{minipage}{.2\linewidth}
  \footnotesize
  \centering
  $\parallelogramsurface{0}{8}{2}$

  {Polyvectors \\ of $\Bl_8\mathbb{P}^2$}
\end{minipage}\end{flushleft}
 
\newpage

\printbibliography

\end{document}